\documentclass[12pt]{article}
\usepackage[utf8]{inputenc}
\usepackage{geometry}
\usepackage{graphicx}
\usepackage{appendix}
\usepackage{graphicx}
\usepackage[natural]{xcolor}
\usepackage{pgf,tikz,pgfplots}
\usepackage{tikz-3dplot}
\usepackage{pifont}
\usepackage{refcount}
\title{Existence of $2$-Factors in Tough Graphs without Forbidden Subgraphs}

\author{ 
    {\large Elizabeth Grimm$^1$, Anna Johnsen$^2$, Songling Shan$^1$}\smallskip \vspace{-.2cm} \\ 
	\medskip  \normalsize{$^1$Department of Mathematics, Illinois State University, Normal, IL 61790}\vspace{-.33cm}\\ 
	\medskip  \normalsize{$^2$Department of Mathematics and Statistics, Georgia State University, Atlanta, GA 30303}\vspace{-.33cm}\\ 
	{\normalsize {\tt evgrimm@ilstu.edu, ajohnsen2@student.gsu.edu, sshan12@ilstu.edu}}
}

\date{\today}
\usepackage[shortlabels]{enumitem}
\usepackage{amssymb}
\usepackage{amsfonts}
\usepackage{amsthm}
\usepackage{graphicx}
\usepackage{amsmath}
\usepackage{indentfirst}
\newtheorem{THM}{Theorem}

\newtheorem{LEM}[THM]{Lemma}
\newtheorem{REM}[THM]{Remark}
\newtheorem*{theorem*}{Theorem}
\newtheorem{CLAIM}{Claim}
\newcommand{\CC}{\mathcal{C}}
\newcommand{\pf}{\textbf{Proof}.\quad}
\usepackage[
pdfauthor={},
pdftitle={},
pdfstartview=XYZ,
bookmarks=true,
colorlinks=true,
linkcolor=blue,
urlcolor=blue,
citecolor=blue,
bookmarks=false,
linktocpage=true,
hyperindex=true
]{hyperref}
\begin{document}
\newcommand{\oC}{\overset{\rightharpoonup }{C}}
\newcommand{\iC}{\overset{\leftharpoonup }{C}}
\providecommand{\keywords}[1]
{
  \small	
  \textbf{\textit{Keywords: }} #1
}

\maketitle

\begin{abstract}
For a given graph $R$, 
a graph $G$ is $R$-free if $G$ does not contain $R$ as an induced subgraph. It is known that every $2$-tough graph with at least three vertices has a $2$-factor. In graphs with restricted structures, it was shown that every $2K_2$-free $3/2$-tough graph  with at least three vertices has a $2$-factor, and the toughness bound $3/2$
is best possible. In viewing $2K_2$, the disjoint union of two edges, as a linear forest, 
in this paper, for any linear forest $R$
on  5, 6, or 7 vertices, we find the sharp toughness bound $t$ such that every $t$-tough $R$-free graph on at least three vertices has a 2-factor.  
%
%
\end{abstract}


\section{Introduction} \label{section:intro}
Let $G$ be a simple, undirected graph and let $V(G)$ and $E(G)$ denote the vertex set and the edge set of $G$, respectively.
We denote the set of neighbors of a vertex $x\in V(G)$ by $N_{G}(x)$. The closed neighborhood of a vertex $x$ in $G$, denoted by $N_G[x]$, is the set $\{x\} \cup
N_G(x)$. For any subset $S \subseteq V(G)$, $G[S]$ is the subgraph of $G$ induced by $S$, $G - S$ denotes the subgraph $G[V(G) \setminus S]$, and $N_G(S) = \bigcup_{v \in S} N_{G}(v)$. Given disjoint subsets $S$ and $T$ of $V(G)$, we denote by $E_{G}(S,T)$ the set of edges which have one end vertex in $S$ and the other end vertex in $T$, and let $e_G(S,T)=|E_G(S,T)|$.
If  $S=\{s\}$ is a singleton,
we write $e_G(s, T)$ for $e_G(\{s\}, T)$.
If $H\subseteq G$ is a subgraph of $G$, and $T\subseteq V(G)$ with
$T\cap V(H)=\emptyset$, we write $E_G(H,T)$ and $e_G(H,T)$
for notational simplicity.

For a given graph $R$, we say that $G$ is $R$-free if there does not exist an induced copy of $R$ in $G$. 
For integers $a$ and $b$ with $a\ge 0$ and $b\ge 1$, we denote by $aP_b$ the graph consisting of $a$ disjoint copies of the path $P_b$. When $a=1$, $1P_b$ is just $P_b$, and when $a=0$, $0P_b$ is the null graph. 
For two integers $p$ and $q$, let $[p,q]=\{i\in \mathbb{Z}: p\le  i \le q\}$.

Denote by $c(G)$ the number of components of $G$. 
Let $t \geq 0$ be a real number. We say a graph $G$ is $t$-tough if for each cutset $S$ of $G$ we have $t \cdot c(G-S) \leq |S|$. The toughness of a graph $G$, denoted $\tau(G)$, is the maximum value of $t$ for which $G$ is $t$-tough if $G$ is non-complete and is defined to be $\infty$ if $G$ is complete. 



For an integer $k\ge 1$, a $k$-regular spanning subgraph is a \emph{$k$-factor} of $G$. It is well known, according to a theorem by Enomoto, Jackson, Katerinis, and  Saito~\cite{ENOMOTO1998277} from 1998,  that every $k$-tough graph with at least three vertices has a $k$-factor  if $k|V(G)|$ is even and $|V(G)|\ge k+ 1$. In terms of a sharp toughness bound, 
particular research interest has been taken when $k=2$ for graphs with restricted structures.  For example, it was shown that every $3/2$-tough 5-chordal graph (graphs with no induced cycle of length at least 5) on at least three vertices has a 2-factor~\cite{5chordal} and that every $3/2$-tough $2K_2$-free graph on at least three vertices has a 2-factor~\cite{Ota2021Toughness2A}. 
The toughness bound $3/2$ is best possible in both results.

A \emph{linear forest} is a graph consisting of disjoint paths. In viewing $2K_2$
as a linear forest on 4 vertices and the result by Ota and Sanka~\cite{Ota2021Toughness2A} that every $3/2$-tough $2K_2$-free graph on at least three vertices has a 2-factor, we investigate  the existence of 2-factors in $R$-free graphs when $R$
is a linear forest on 5, 6, or 7 vertices. These graphs  $R$ are listed below, where the unions are 
vertex disjoint unions. 
\begin{flushright}
\begin{enumerate}
	\item \quad  $P_5$ \quad  $P_4\cup P_1$\quad  $P_3\cup P_2$ \quad  $P_3\cup 2P_1$ \quad  $2P_2\cup P_1$ \quad  $P_2\cup 3P_1 \quad 5P_1$; 
	\item \quad $P_6$ \quad $P_5\cup P_1$ \quad $P_4\cup P_2$ \quad $P_4\cup 2P_1$\quad $2P_3$ \quad $P_3\cup P_2\cup P_1$\quad $P_3\cup 3P_1$ \quad $3P_2$   \quad $2P_2\cup 2P_1$ \quad $P_2\cup 4P_1$\quad $6P_1$; 
	\item \quad  $P_7$ \quad $P_6\cup P_1$ \quad $P_5\cup P_2$ \quad $P_5\cup 2P_1$ \quad $P_4\cup P_3$ \quad $P_4\cup P_2\cup P_1$ \quad $P_4 \cup 3P_1$ \quad $2P_3 \cup P_1$ \quad $P_3\cup 2P_2$ \quad $P_3 \cup P_2\cup 2P_1$ \quad $P_3 \cup 4P_1$ \quad $3P_2\cup P_1$ \quad $2P_2\cup 3P_1$ \quad $P_2\cup 5P_1$ \quad $7P_1$.
\end{enumerate} 
\end{flushright}

Our main results are the following:
\begin{THM}\label{theorem1}
	Let $t>0$ be a real number, $R$ be any linear forest on $5$ vertices, and $G$ be a $t$-tough $R$-free graph on at least 3 vertices. Then $G$ has a 2-factor provided that 
	\begin{enumerate}[(1)]
		\item $R \in \{P_4 \cup P_1, P_3\cup 2P_1, P_2\cup 3P_1\}$  and $t=1$ unless
		\begin{enumerate}
			\item $R = P_2 \cup 3P_1$, and $G \cong H_0$ or $G$ contains $H_1$, $H_2$ or $H_3$ as a spanning subgraph such that $E(G)\setminus E(H_i) \subseteq E_G(S, V(G)\setminus (T\cup S))$ for each $i\in [1,3]$, where $H_i$, $S$ and $T$ are defined in Figure~\ref{f1}.   
			\item $R=P_3 \cup 2P_1$ and $G$ contains $H_1$ as a spanning subgraph such that $E(G)\setminus E(H_1) \subseteq E_G(S, V(G)\setminus (T\cup S))$.   
		\end{enumerate}
		\item $R = 5P_1$ and $t>1$. 
		\item $R\in \{P_5,  P_3\cup P_2, 2P_2\cup P_1\}$ and $t=3/2$. 
		\end{enumerate}
\end{THM}

\begin{figure}[!htb]
\begin{tikzpicture}[scale=1]
	\begin{scope}[scale=0.5]

		\coordinate (D1) at (-4,4.5);
		\coordinate (D2) at (4,4.5);
		\coordinate (T) at (0,0);
		\node (d1) at (-4,5) {$D_1$};
		\node (d2) at (4,5) {$D_2$};
		\node (t) at (0,-0.5) {$T$};
		
		\coordinate (v1) at (-5,3.75);
		\coordinate (v2) at (-4,3.5);
		\coordinate (v3) at (-3,3.75);
		\coordinate (v4) at (3,3.75);
		\coordinate (v5) at (4,3.5);
		\coordinate (v6) at (5,3.75);
		\coordinate (v7) at (-2,0.6);
		\coordinate (v8) at (0,0.6);
		\coordinate (v9) at (2,0.6);

		\node[circle,fill=black,inner sep=1pt,label=$v_1$] at (v1) {};
		\node[circle,fill=black,inner sep=1pt, label=$v_2$] at (v2) {};
		\node[circle,fill=black,inner sep=1pt, label=$v_3$] at (v3) {};
		
		\node[circle,fill=black,inner sep=1pt, label=$v_4$] at (v4) {};
		\node[circle,fill=black,inner sep=1pt, label=$v_5$] at (v5) {};
		\node[circle,fill=black,inner sep=1pt, label=$v_6$] at (v6) {};
		
		\node[circle,fill=black,inner sep=1pt, label=$t_1$] at (v7) {};
		\node[circle,fill=black,inner sep=1pt, label=$\,\,\,t_2$] at (v8) {};
		\node[circle,fill=black,inner sep=1pt, label=$ t_3$] at (v9) {};
		
		
		\draw (v1) to (v7) to (v4);
		\draw (v2) to (v8) to (v5);
		\draw (v3) to (v9) to (v6);
		\draw (v1) to (v2) to (v3) to (v1);
		\draw (v4) to (v5) to (v6) to (v4);
		\node () at (0,-2.5) {The graph $H_0$};
			\end{scope}
		
		\begin{scope}[shift={(8,0)},scale=0.5]
		\coordinate (S) at (-4,4);
		\coordinate (D2) at (4,4);
		\coordinate (Da) at (4,4.75);
		\coordinate (T) at (0,0);
		\coordinate (Tb) at (0,-0.5);
		\node (d2) at (4,5) {$D$};
		\node (t) at (0,-1) {$T$};
		\node (s) at (-4.5,4)  {$x$};
		
		\coordinate (v4) at (2.5,3.25);
		\coordinate (v5) at (4.5,4);
		\coordinate (v6) at (5.5,3.25);
		\coordinate (v7) at (-2,0.25);
		\coordinate (v8) at (0,0.25);
		\coordinate (v9) at (2,0.25);

		\node[circle,fill=black,inner sep=1pt,label=$S$] at (S) {};

		\node[circle,fill=black,inner sep=1pt, label=$v_1$] at (v4) {};
		\node[circle,fill=black,inner sep=1pt, label=left:$v_2$] at (v5) {};
		\node[circle,fill=black,inner sep=1pt, label=$v_3$] at (v6) {};
		
		\node[circle,fill=black,inner sep=1pt, label=$t_1$] at (v7) {};
		\node[circle,fill=black,inner sep=1pt, label=$t_2$] at (v8) {};
		\node[circle,fill=black,inner sep=1pt, label=$t_3$] at (v9) {};
		
		
		\draw (S) to (v7) to (v4);
		\draw (S) to (v8) to (v5);
		\draw (S) to (v9) to (v6);
		\draw (v4) to (v5) to (v6) to (v4);
		\node () at (0,-2.5) {The graph $H_1$};
		\end{scope}
	
	\begin{scope}[shift={(0,-5)}, scale=0.4]
	\coordinate (S) at (-4,4);
	\coordinate (D2) at (4,4);
	\coordinate (Da) at (4,4.75);
	\coordinate (T) at (0,0);
	\coordinate (Tb) at (0,-0.5);
	\node (d2) at (4,5) {$D$};
	\node (t) at (0, -1) {$T$};
	\node (s) at (-4.5,4)  {$x$};
	
	\coordinate (v3) at (3.25,3.15);
	\coordinate (v4) at (4.75,3.15);
	\coordinate (v5) at (2.5,4);
	\coordinate (v6) at (5.5,4);
	\coordinate (v7) at (-2,0.25);
	\coordinate (v8) at (0,0.25);
	\coordinate (v9) at (2,0.25);
	
	\node[circle,fill=black,inner sep=1pt,label=$S$] at (S) {};
	
	\node[circle,fill=black,inner sep=1pt, label=left:$v_1$] at (v3) {};
	\node[circle,fill=black,inner sep=1pt, label=right:$v_2$] at (v4) {};
	\node[circle,fill=black,inner sep=1pt, label=$v_3$] at (v5) {};
	\node[circle,fill=black,inner sep=1pt, label=$v_4$] at (v6) {};
	
	\node[circle,fill=black,inner sep=1pt, label=$t_1$] at (v7) {};
	\node[circle,fill=black,inner sep=1pt, label=$t_2$] at (v8) {};
	\node[circle,fill=black,inner sep=1pt, label=$t_3$] at (v9) {};
	
	
	\draw (S) to (v7) to (v5) to (v3);
	\draw (S) to (v8) to (v3);
	\draw (S) to (v9) to (v4) to (v6);
	\draw (v3) to (v4) to (v5) to (v6) to (v3);
		\node () at (0,-2.5) {The graph $H_2$};
	\end{scope}

\begin{scope}[shift={(5,-5)}, scale=0.4]
\coordinate (S) at (-4,4);
\coordinate (D2) at (4,4);
\coordinate (Da) at (4,4.75);
\coordinate (T) at (0,0);
\coordinate (Tb) at (0,-0.5);
\node (d2) at (4,5) {$D$};
\node (t) at (0,-1) {$ T$};
\node (s) at (-4.5,4)  {$x$};

\coordinate (v3) at (3.25,3.15);
\coordinate (v4) at (4.75,3.15);
\coordinate (v5) at (2.5,4);
\coordinate (v6) at (5.5,4);
\coordinate (v7) at (-2,0.25);
\coordinate (v8) at (0,0.25);
\coordinate (v9) at (2,0.25);

\node[circle,fill=black,inner sep=1pt,label=$S$] at (S) {};

\node[circle,fill=black,inner sep=1pt, label=left:$v_1$] at (v3) {};
\node[circle,fill=black,inner sep=1pt, label=right:$v_2$] at (v4) {};
\node[circle,fill=black,inner sep=1pt, label=$v_3$] at (v5) {};
\node[circle,fill=black,inner sep=1pt, label=$v_4$] at (v6) {};

\node[circle,fill=black,inner sep=1pt, label=$t_1$] at (v7) {};
\node[circle,fill=black,inner sep=1pt, label=$t_2$] at (v8) {};
\node[circle,fill=black,inner sep=1pt, label=$t_3$] at (v9) {};

\draw (S) to (v7) to (v5) to (v3);
\draw (S) to (v8) to (v3);
\draw (S) to (v9) to (v4) to (v6);
\draw (v3) to (v4);
\draw (v5) to (v6) to (v3);
\node () at (0,-2.5) {The graph $H_3$};
\end{scope}

\begin{scope}[shift={(10,-5)}, scale=0.4]
	\coordinate (S) at (-4,4);
	\coordinate (D2) at (4,4);
	\coordinate (Da) at (4,4.75);
	\coordinate (T) at (0,0);
	\coordinate (Tb) at (0,-0.5);
	\node (d2) at (4,5) {$D$};
	\node (t) at (0,-1) {$ T$};
	\node (s) at (-4.5,4)  {$x$};
	
	\coordinate (v3) at (3.25,3.15);
	\coordinate (v4) at (4.75,3.15);
	\coordinate (v5) at (2.5,4);
	\coordinate (v6) at (5.5,4);
	\coordinate (v7) at (-2,0.25);
	\coordinate (v8) at (0,0.25);
	\coordinate (v9) at (2,0.25);
	
	\node[circle,fill=black,inner sep=1pt,label=$S$] at (S) {};
	
	\node[circle,fill=black,inner sep=1pt, label=left:$v_1$] at (v3) {};
	\node[circle,fill=black,inner sep=1pt, label=right:$v_2$] at (v4) {};
	\node[circle,fill=black,inner sep=1pt, label=$v_3$] at (v5) {};
	\node[circle,fill=black,inner sep=1pt, label=$v_4$] at (v6) {};
	
	\node[circle,fill=black,inner sep=1pt, label=$t_1$] at (v7) {};
	\node[circle,fill=black,inner sep=1pt, label=$t_2$] at (v8) {};
	\node[circle,fill=black,inner sep=1pt, label=$t_3$] at (v9) {};

	\draw (S) to (v7) to (v5);
	\draw (S) to (v8) to (v3);
	\draw (S) to (v9) to (v4) to (v6);
	\draw (v3) to (v4);
	\draw (v5) to (v6) to (v3);
	\node () at (0,-2.5) {The graph $H_4$};
\end{scope}
\end{tikzpicture}
\caption{The four exceptional graphs for Theorem~\ref{theorem1}(1), where  $S=\{x\}$ and $T=\{t_1,t_2,t_3\}$.}
\label{f1}
\end{figure}
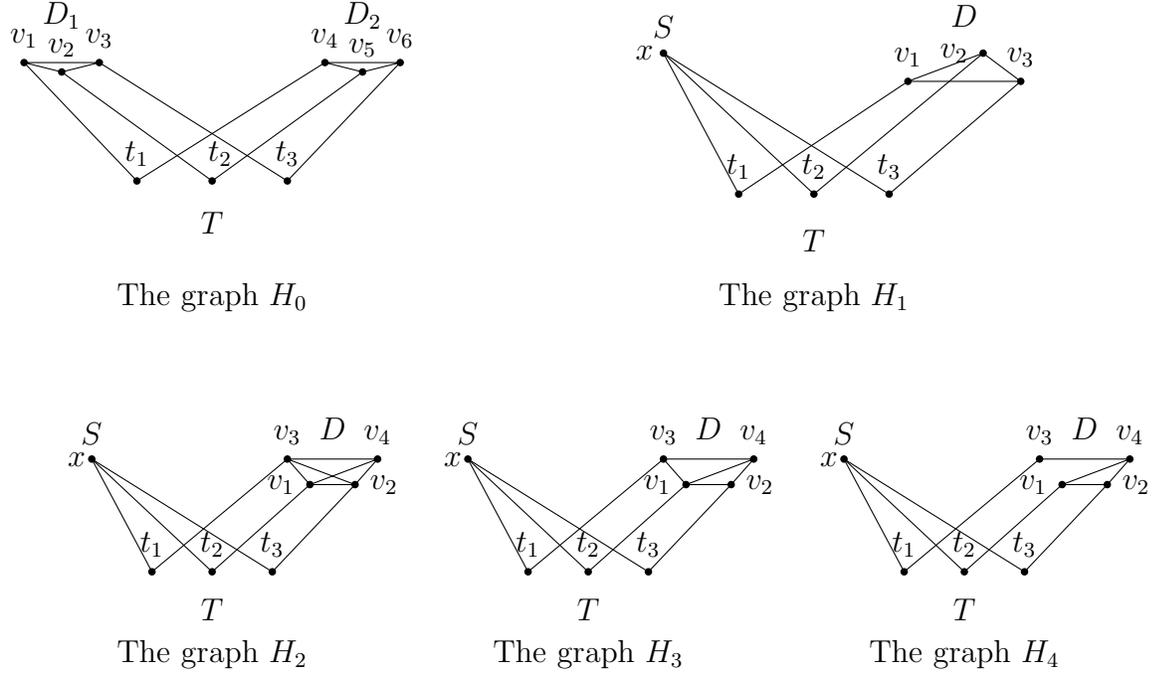

\begin{THM}\label{theorem2}
	Let $t>0$ be a real number, $R$ be any linear forest on $6$ vertices, and $G$ be a $t$-tough $R$-free graph on at least 3 vertices. Then $G$ has a 2-factor provided that 
	\begin{enumerate}[(1)]
		\item $R\in \{P_4\cup 2P_1, P_3\cup 3P_1, P_2\cup 4P_1, 6P_1\}$ and $t>1$ unless $R=6P_1$ and $G$ contains $H_5$ with $p=5$ as a spanning subgraph such that $E(G)\setminus E(H_5) \subseteq E_G(S, V(G)\setminus (T\cup S))$, where $H_5$, $S$ and $T$ are defined in Figure~\ref{f2}.  
		\item $R\in \{P_6,  P_5\cup P_1, P_4\cup P_2, 2P_3, P_3\cup P_2\cup P_1, 3P_2, 2P_2\cup 2P_1\}$ and $t=3/2$. 
	\end{enumerate}
\end{THM}

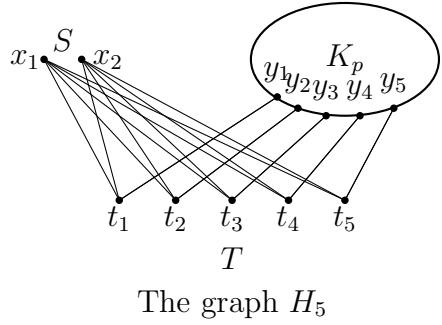
\begin{figure}[!htb]
	\begin{center}
		
		\begin{tikzpicture}[scale=1]
			
			\begin{scope}[scale=0.5]

				\coordinate (s1) at (-5,4);
				\coordinate (s2) at (-4,4);
				\coordinate (S) at (-4.5,4);
				\coordinate (sl) at (-4.5,4.5);
				\coordinate (D) at (3,4);
				\coordinate (Da) at (3,4);
				\coordinate (T) at (0,0);
				\coordinate (Tb) at (0,-1.3);
				\node (d2) at (Da) {$K_p$};
				\node (t) at (Tb) {$T$};
				\node (SL) at (sl) {$S$};
				\node (SL) at (-5.5,4) {$x_1$};
				\node (SL) at (-3.3,4) {$x_2$};
				
				\coordinate (v1) at (-3,0.25);
				\coordinate (v2) at (-1.5,0.25);
				\coordinate (v3) at (0,0.25);
				\coordinate (v4) at (1.5,0.25);
				\coordinate (v5) at (3,0.25);
				
				\node (t) at (-3,0.25-0.5) {$t_1$}; 
				\node (t) at (-1.5,0.25-0.5) {$t_2$};
				\node (t) at (0,0.25-0.5) {$t_3$};
				\node (t) at (1.5,0.25-0.5) {$t_4$};
				\node (t) at (3,0.25-0.5) {$t_5$};
				
				\coordinate (v6) at (1.2,3);
				\coordinate (v7) at (1.75,2.7);
				\coordinate (v8) at (2.5,2.5);
				\coordinate (v9) at (3.4,2.5);
				\coordinate (v0) at (4.3,2.7);

				\node[circle,fill=black,inner sep=1pt,label=] at (s1) {};
				\node[circle,fill=black,inner sep=1pt,label=] at (s2) {};

				\node[circle,fill=black,inner sep=1pt, label=] at (v1) {};
				\node[circle,fill=black,inner sep=1pt, label=] at (v2) {};
				\node[circle,fill=black,inner sep=1pt, label=] at (v3) {};
				\node[circle,fill=black,inner sep=1pt, label=] at (v4) {};
				\node[circle,fill=black,inner sep=1pt, label=] at (v5) {};
				\node[circle,fill=black,inner sep=1pt, label=$y_1$] at (v6) {};
				\node[circle,fill=black,inner sep=1pt, label=$y_2$] at (v7) {};
				\node[circle,fill=black,inner sep=1pt, label=$y_3$] at (v8) {};
				\node[circle,fill=black,inner sep=1pt, label=$y_4$] at (v9) {};
				\node[circle,fill=black,inner sep=1pt, label=$y_5$] at (v0) {};
				
				\draw[thick] (D) ellipse (2.5 and 1.5);
				
				\draw (s1) to (v1) to (v6);
				\draw (s1) to (v2) to (v7);
				\draw (s1) to (v3) to (v8);
				\draw (s1) to (v4) to (v9);
				\draw (s1) to (v5) to (v0);
				
				\draw (s2) to (v1) to (v6);
				\draw (s2) to (v2) to (v7);
				\draw (s2) to (v3) to (v8);
				\draw (s2) to (v4) to (v9);
				\draw (s2) to (v5) to (v0);
				\node () at (0,-2.5) {The graph $H_5$};
			\end{scope}

		\end{tikzpicture}
		
	\end{center}
	\caption{The  exceptional graph for Theorem~\ref{theorem2}(1), where $S=\{x_1,x_2\}$, $T=\{t_1,\ldots, t_5\}$,  and $p=5$.}
	\label{f2}
\end{figure}

\begin{THM}\label{theorem3}
	Let $t>0$ be a real number, $R$ be any linear forest on $7$ vertices, and $G$ be a $t$-tough $R$-free graph on at least 3 vertices. Then $G$ has a 2-factor provided that 
	\begin{enumerate}[(1)]
		\item $R\in \{P_4\cup 3P_1, P_3\cup 4P_1, P_2\cup 5P_1\}$ and $t>1$ unless 
		\begin{enumerate}
			\item  when $R\ne P_4\cup 3P_1$, $G$ contains $H_5$ with $p=5$ as a spanning subgraph such that $E(G)\setminus E(H_5) \subseteq E_G(S, V(G)\setminus (T\cup S)) \cup E(G[S])$, where $H_5$, $S$ and $T$ are defined in Figure~\ref{f2}.  
			\item $R=P_2 \cup 5P_1$ and $G$ contains one of  $H_6, \ldots, H_{11}$ as a spanning subgraph such that $E(G)\setminus E(H_i) \subseteq E_G(S, V(G)\setminus (T\cup S)) \cup E(G[S]) \cup  E(G[ V(G)\setminus (T\cup S)])$,    where $H_i$, $S$ and $T$ are defined in Figure~\ref{f3} for each $i\in [6,11]$. 
		\end{enumerate}
		
		\item $R=7P_1$ and $t>\frac{7}{6}$ unless $G$ contains $H_5$ with $p=5$ as a spanning subgraph such that $E(G)\setminus E(H_5) \subseteq E_G(S, V(G)\setminus (T\cup S)) \cup E(G[S])$. 
		\item $R\in \{P_7, P_6\cup P_1, P_5\cup P_2, P_5\cup 2P_1, P_4\cup P_2\cup P_1, 2P_3\cup P_1,P_4\cup P_3, P_3\cup 2P_2, P_3\cup P_2\cup 2P_1, 3P_2\cup P_1, 2P_2\cup 3P_1\}$ and $t=3/2$. 
	\end{enumerate}
\end{THM}

\begin{figure}[!htb]
\begin{tikzpicture}[scale=1]
	\begin{scope}[scale=0.4] 
	\node (D) at (4,6) {$D$};
	\node (T) at (2,-0.75) {$T$};
	\node (S) at (-4.5,6) {$S$};
	\node (p) at (-4.25, 3) {$+$};
	
	\coordinate (v1) at (-2,3);
	\coordinate (v2) at (0,3);
	\coordinate (v3) at (2,3);
	\coordinate (v0) at (2,5);
	\coordinate (v4) at (4,3);
	\coordinate (v5) at (6,3);
	
	\node (v1l) at (-2,3.75) {$v_1$};
	\node (v2l) at (-0.75,3) {$v_2$};
	\node (v3l) at (1.5,3.5) {$v_3$};
	\node (v4l) at (3,3.5) {$v_4$};
	\node (v5l) at (6.25,3.75) {$v_5$};
	
	\coordinate (t1) at (-2,0.75);
	\coordinate (t2) at (0,0.75);
	\coordinate (t3) at (2,0.75);
	\coordinate (t4) at (4,0.75);
	\coordinate (t5) at (6,0.75);
	
	\node (t1l) at (-2.5,0.5) {$t_1$};
	\node (t2l) at (-.5,0.5) {$t_2$};
	\node (t3l) at (1.5,0.5) {$t_3$};
	\node (t4l) at (3.5,0.5) {$t_4$};
	\node (t5l) at (5.5,0.5) {$t_5$};
	
	\coordinate (s2) at (-4,4.25);
	\coordinate (s1) at (-5,4.25);
	
	\node[circle,fill=black,inner sep=1pt,label=$v_0$] at (v0) {};		
	\node[circle,fill=black,inner sep=1pt,label=] at (v1) {};
	\node[circle,fill=black,inner sep=1pt, label=] at (v2) {};
	\node[circle,fill=black,inner sep=1pt, label=] at (v3) {};
	\node[circle,fill=black,inner sep=1pt, label=] at (v4) {};
	\node[circle,fill=black,inner sep=1pt, label=] at (v5) {};
		
	\node[circle,fill=black,inner sep=1pt, label=] at (t1) {};
	\node[circle,fill=black,inner sep=1pt, label=] at (t2) {};
	\node[circle,fill=black,inner sep=1pt, label=] at (t3) {};
	\node[circle,fill=black,inner sep=1pt, label=] at (t4) {};
	\node[circle,fill=black,inner sep=1pt, label=] at (t5) {};
		
	\node[circle,fill=black,inner sep=1pt, label= $x_1$] at (s1) {};
	\node[circle,fill=black,inner sep=1pt, label=$x_2$] at (s2) {};
		
	\foreach \i in {v1,v2,v3,v4,v5}
        \draw (v0) to (\i); 

    \draw (s2) to (v1);
    \draw (v2) to (v3) to (v4) to (v5) to[out=220,in=320] (v2);
    \draw (v5) to[out=210,in=330] (v3);
    \draw (v4) to[out=210,in=330] (v2);
    
	\draw (v1) to (t1);
	\draw (v2) to (t2);
	\draw (v3) to (t3);
	\draw (v4) to (t4);
	\draw (v5) to (t5);
		\node () at (2,-2.5) {The graph $H_6$};
			\end{scope}
		
		\begin{scope}[shift={(5,0)},scale=0.4] 
	\node (D) at (5,7) {$D$};
	\node (T) at (3,-0.5) {$T$};
	\node (S) at (-3.5,6) {$S$};
	\node (p) at (-3.25, 3) {$+$};
	
	\coordinate (v1) at (-1,3);
	\coordinate (v2) at (1,3);
	\coordinate (v3) at (3,3);
	\coordinate (v0) at (3,6);
	\coordinate (v4) at (5,3);
	\coordinate (v5) at (7,3);
	
	\node (v0l) at (3,6.5) {$v_0$};
	\node (v1l) at (-1.75,3) {$v_1$};
	\node (v2l) at (1.5,2.5) {$v_2$};
	\node (v3l) at (3.5,3.5) {$v_3$};
	\node (v4l) at (5.5,3.5) {$v_4$};
	\node (v5l) at (7.5,3.5) {$v_5$};
	
	\coordinate (t1) at (-1,0.75);
	\coordinate (t2) at (1,0.75);
	\coordinate (t3) at (3,0.75);
	\coordinate (t4) at (5,0.75);
	\coordinate (t5) at (7,0.75);
	
	\node (t1l) at (-1.5,0.5) {$t_1$};
	\node (t2l) at (0.5,0.5) {$t_2$};
	\node (t3l) at (2.5,0.5) {$t_3$};
	\node (t4l) at (4.5,0.5) {$t_4$};
	\node (t5l) at (6.5,0.5) {$t_5$};
	
	\coordinate (s2) at (-3,4.25);
	\coordinate (s1) at (-4,4.25);
	
	\node[circle,fill=black,inner sep=1pt,label=] at (v0) {};		
	\node[circle,fill=black,inner sep=1pt,label=] at (v1) {};
	\node[circle,fill=black,inner sep=1pt, label=] at (v2) {};
	\node[circle,fill=black,inner sep=1pt, label=] at (v3) {};
	\node[circle,fill=black,inner sep=1pt, label=] at (v4) {};
	\node[circle,fill=black,inner sep=1pt, label=] at (v5) {};
		
	\node[circle,fill=black,inner sep=1pt, label=] at (t1) {};
	\node[circle,fill=black,inner sep=1pt, label=] at (t2) {};
	\node[circle,fill=black,inner sep=1pt, label=] at (t3) {};
	\node[circle,fill=black,inner sep=1pt, label=] at (t4) {};
	\node[circle,fill=black,inner sep=1pt, label=] at (t5) {};
		
	\node[circle,fill=black,inner sep=1pt, label=$x_1$] at (s1) {};
	\node[circle,fill=black,inner sep=1pt, label=$x_2$] at (s2) {};
		
	\foreach \i in {v1,v2,v3,v4,v5}
        \draw (v0) to (\i); 

    \draw (v1) to (v2);
    \draw (v3) to (v4) to (v5) to[out=210,in=330] (v3);
    
	\draw (v1) to (t1);
	\draw (v2) to (t2);
	\draw (v3) to (t3);
	\draw (v4) to (t4);
	\draw (v5) to (t5);
	\node () at (3,-2.5) {The graph $H_7$};
\end{scope}
		
		\begin{scope}[shift={(10,0)},scale=0.4] 
	\node (D) at (4,7) {$D$};
	\node (T) at (4,-0.75) {$T$};
	\node (S) at (-2.5,6) {$S$};
	\node (p) at (-2., 3) {$+$};
	
	\coordinate (v1) at (0,3);
	\coordinate (v2) at (2,4);
	\coordinate (v0) at (2,6);
	\coordinate (v3) at (4,3);
	\coordinate (v4) at (6,3);
	\coordinate (v5) at (8,3);
	
	\node (v0l) at (-0.25,3.75) {$v_2$};
	\node (v1l) at (1.5,4.3) {$v_1$};
	\node (v2l) at (2,6.5) {$v_0$};
	\node (v3l) at (4.3,3.5) {$v_3$};
	\node (v4l) at (6,3.55) {$v_4$};
	\node (v5l) at (8.5,3.25) {$v_5$};
	
	\coordinate (t1) at (0,0.75);
	\coordinate (t2) at (2,0.75);
	\coordinate (t3) at (4,0.75);
	\coordinate (t4) at (6,0.75);
	\coordinate (t5) at (8,0.75);
	
	\node (t1l) at (-0.5,0.5) {$t_1$};
	\node (t2l) at (1.5,0.5) {$t_2$};
	\node (t3l) at (3.5,0.5) {$t_3$};
	\node (t4l) at (5.5,0.5) {$t_4$};
	\node (t5l) at (7.5,0.5) {$t_5$};
	
	\coordinate (s2) at (-2,4.25);
	\coordinate (s1) at (-3,4.25);
	
	\node[circle,fill=black,inner sep=1pt,label=] at (v0) {};		
	\node[circle,fill=black,inner sep=1pt,label=] at (v1) {};
	\node[circle,fill=black,inner sep=1pt, label=] at (v2) {};
	\node[circle,fill=black,inner sep=1pt, label=] at (v3) {};
	\node[circle,fill=black,inner sep=1pt, label=] at (v4) {};
	\node[circle,fill=black,inner sep=1pt, label=] at (v5) {};
		
	\node[circle,fill=black,inner sep=1pt, label=] at (t1) {};
	\node[circle,fill=black,inner sep=1pt, label=] at (t2) {};
	\node[circle,fill=black,inner sep=1pt, label=] at (t3) {};
	\node[circle,fill=black,inner sep=1pt, label=] at (t4) {};
	\node[circle,fill=black,inner sep=1pt, label=] at (t5) {};
		
	\node[circle,fill=black,inner sep=1pt, label=$x_1$] at (s1) {};
	\node[circle,fill=black,inner sep=1pt, label=$x_2$] at (s2) {};
		
	\foreach \i in {v1,v2,v3,v4,v5}
        \draw (v0) to (\i); 
        
    \foreach \i in {v1,v3}
        \draw (v2) to (\i);
        
    \draw (s1) to (v1);
    \draw (v3) to (v4) to (v5) to[out=210,in=330] (v3);
	\draw (v1) to (t1);
	\draw (v2) to (t2);
	\draw (v3) to (t3);
	\draw (v4) to (t4);
	\draw (v5) to (t5);
	\node () at (4,-2.5) {The graph $H_8$};
\end{scope}
	
	\begin{scope}[shift={(-2,-5)}, scale=0.4] 
	\node (D) at (9,7) {$D$};
	\node (T) at (7,-0.75) {$T$};
	\node (S) at (0.5,6) {$S$};
	\node (p) at (0.75, 3) {$+$};
	
	\coordinate (v0) at (7,6.5);
	\coordinate (v1) at (3,3);
	\coordinate (v2) at (5,3);
	\coordinate (v3) at (7,4);
	\coordinate (v4) at (9,3);
	\coordinate (v5) at (11,3);
	
	\node (v0l) at (7,7) {$v_0$};
	\node (v1l) at (2.25,3) {$v_2$};
	\node (v2l) at (5.5,2.5) {$v_3$};
	\node (v3l) at (7.5,4.5) {$v_1$};
	\node (v4l) at (8.5,2.5) {$v_4$};
	\node (v5l) at (11.75,3) {$v_5$};
	
	\coordinate (t1) at (3,0.75);
	\coordinate (t2) at (5,0.75);
	\coordinate (t3) at (7,0.75);
	\coordinate (t4) at (9,0.75);
	\coordinate (t5) at (11,0.75);
	
	\node (t1l) at (2.5,0.5) {$t_1$};
	\node (t2l) at (4.5,0.5) {$t_2$};
	\node (t3l) at (6.5,0.5) {$t_3$};
	\node (t4l) at (8.5,0.5) {$t_4$};
	\node (t5l) at (10.5,0.5) {$t_5$};
	
	\coordinate (s2) at (1,4.25);
	\coordinate (s1) at (0,4.25);
	
	\node[circle,fill=black,inner sep=1pt,label=] at (v0) {};		
	\node[circle,fill=black,inner sep=1pt,label=] at (v1) {};
	\node[circle,fill=black,inner sep=1pt, label=] at (v2) {};
	\node[circle,fill=black,inner sep=1pt, label=] at (v3) {};
	\node[circle,fill=black,inner sep=1pt, label=] at (v4) {};
	\node[circle,fill=black,inner sep=1pt, label=] at (v5) {};
		
	\node[circle,fill=black,inner sep=1pt, label=] at (t1) {};
	\node[circle,fill=black,inner sep=1pt, label=] at (t2) {};
	\node[circle,fill=black,inner sep=1pt, label=] at (t3) {};
	\node[circle,fill=black,inner sep=1pt, label=] at (t4) {};
	\node[circle,fill=black,inner sep=1pt, label=] at (t5) {};
		
	\node[circle,fill=black,inner sep=1pt, label=$x_1$] at (s1) {};
	\node[circle,fill=black,inner sep=1pt, label=$x_2$] at (s2) {};
		
	\foreach \i in {v1,v2,v3,v4,v5}
        \draw (v0) to (\i); 
        
    \foreach \i in {v1,v2,v4,v5}
        \draw (v3) to (\i);
        
    \draw (v1) to (v2) to (v3) to (v4) to (v5);
	\draw (v1) to (t1);
	\draw (v2) to (t2);
	\draw (v3) to (t3);
	\draw (v4) to (t4);
	\draw (v5) to (t5);
	\node () at (7,-2.5) {The graph $H_9$};
	\end{scope}

 \begin{scope}[shift={(5.5,-5)}, scale=0.4] 
 	\node (D) at (2,7) {$D$};
 	\node (T) at (2,-0.8) {$T$};
 	\node (S) at (-4.5,6) {$S$};
 	\node (p) at (-4.5, 3) {$+$};
	
 	\coordinate (v1) at (-2,3);
 	\coordinate (v2) at (0,3);
 	\coordinate (v3) at (2,3);
 	\coordinate (v0) at (2,5);
 	\coordinate (v4) at (4,3);
 	\coordinate (v5) at (6,3);
	
 	\coordinate (t1) at (-2,0.75);
 	\coordinate (t2) at (0,0.75);
 	\coordinate (t3) at (2,0.75);
 	\coordinate (t4) at (4,0.75);
 	\coordinate (t5) at (6,0.75);
 	
 	\node (t1l) at (-2.5,0.25) {$t_1$};
 	\node (t2l) at (-0.5,0.25) {$t_2$};
 	\node (t3l) at (1.75,0.25) {$t_3$};
 	\node (t4l) at (3.5,0.25) {$t_4$};
 	\node (t5l) at (5.5,0.25) {$t_5$};
 	
 	\node (v1l) at (-2, 3.5) {$v_1$};
 	\node (v2l) at (-0.1,3.5) {$v_2$};
 	\node (v3l) at (1.5,3.5) {$v_3$};
 	\node (v4l) at (3,3.5) {$v_4$};
 	\node (v5l) at (5.7,3.5) {$v_5$};
	
 	\coordinate (s2) at (-4,4.25);
 	\coordinate (s1) at (-5,4.25);
	
 	\node[circle,fill=black,inner sep=1pt,label=$v_0$] at (v0) {};		
 	\node[circle,fill=black,inner sep=1pt,label=] at (v1) {};
 	\node[circle,fill=black,inner sep=1pt, label=] at (v2) {};
 	\node[circle,fill=black,inner sep=1pt, label=] at (v3) {};
 	\node[circle,fill=black,inner sep=1pt, label=] at (v4) {};
 	\node[circle,fill=black,inner sep=1pt, label=] at (v5) {};
		
 	\node[circle,fill=black,inner sep=1pt, label=] at (t1) {};
 	\node[circle,fill=black,inner sep=1pt, label=] at (t2) {};
 	\node[circle,fill=black,inner sep=1pt, label=] at (t3) {};
 	\node[circle,fill=black,inner sep=1pt, label=] at (t4) {};
 	\node[circle,fill=black,inner sep=1pt, label=] at (t5) {};
		
 	\node[circle,fill=black,inner sep=1pt, label=$x_1$] at (s1) {};
 	\node[circle,fill=black,inner sep=1pt, label=$x_2$] at (s2) {};
		
 	\foreach \i in {v1,v2,v3,v4,v5}
         \draw (v0) to (\i); 

     \draw (v1) to (v2) to (v3) to (v4) to (v5);
 	\draw (v1) to (t1);
 	\draw (v2) to (t2);
 	\draw (v3) to (t3);
 	\draw (v4) to (t4);
 	\draw (v5) to (t5);
 	\node () at (2.5,-2.5) {The graph $H_{10}$};
 		\draw  (v1) to[out=-20,in=200] (v5);
 \end{scope}

\begin{scope}[shift={(10.5,-5)}, scale=0.4] 
	\node (D1) at (1.5,7) {$D_1$};
	\node (D2) at (6,7) {$D_2$};
	\node (T) at (3.75,-1.) {$T$};
	\node (S) at (-4,6) {$S$};
	\node (p) at (-3.75, 2) {$+$};
	
	\coordinate (v1) at (-0.5,6);
	\coordinate (v2) at (1.5,3);
	\coordinate (v3) at (3,4.75);
	\coordinate (v4) at (4.5,4.75);
	\coordinate (v5) at (6,3.53);
	\coordinate (v6) at (8,6);
	
	\node (v1l) at (-1.25,5.5) {$v_1$};
	\node (v2l) at (0.5,3.5) {$v_2$};
	\node (v3l) at (2.75,5.5) {$v_3$};
	\node (v4l) at (4.5,5.5) {$v_4$};
	\node (v5l) at (6.75,3.5) {$v_5$};
	\node (v6l) at (8.75,5.5) {$v_6$};
	
	\coordinate (t1) at (-0.5,0.5);
	\coordinate (t2) at (1.5,0.5);
	\coordinate (t3) at (3.75,0.5);
	\coordinate (t4) at (6,0.5);
	\coordinate (t5) at (8,0.5);
	
	\node (t1l) at (-1,0.25) {$t_1$};
	\node (t2l) at (1,0.25) {$t_2$};
	\node (t3l) at (3.25,0.25) {$t_3$};
	\node (t4l) at (5.5,0.25) {$t_4$};
	\node (t5l) at (7.5,0.25) {$t_5$};
	
	\coordinate (s2) at (-3.5,3.25);
	\coordinate (s1) at (-4.5,3.25);

	\node[circle,fill=black,inner sep=1pt,label=] at (v1) {};
	\node[circle,fill=black,inner sep=1pt, label=] at (v2) {};
	\node[circle,fill=black,inner sep=1pt, label=] at (v3) {};
		
	\node[circle,fill=black,inner sep=1pt, label=] at (v4) {};
	\node[circle,fill=black,inner sep=1pt, label=] at (v5) {};
	\node[circle,fill=black,inner sep=1pt, label=] at (v6) {};
		
	\node[circle,fill=black,inner sep=1pt, label=] at (t1) {};
	\node[circle,fill=black,inner sep=1pt, label=] at (t2) {};
	\node[circle,fill=black,inner sep=1pt, label=] at (t3) {};
	\node[circle,fill=black,inner sep=1pt, label=] at (t4) {};
	\node[circle,fill=black,inner sep=1pt, label=] at (t5) {};
		
	\node[circle,fill=black,inner sep=1pt, label=$x_1$] at (s1) {};
	\node[circle,fill=black,inner sep=1pt, label=$x_2$] at (s2) {};

	\draw (v3) to (t3) to (v4);
	\draw (v1) to (t1);
	\draw (v2) to (t2);
	\draw (v5) to (t4);
	\draw (v6) to (t5);
	\draw (v1) to (v2) to (v3) to (v1);
	\draw (v4) to (v5) to (v6) to (v4);
	\node () at (4,-2.5) {The graph $H_{11}$};
\end{scope}
\end{tikzpicture}
\caption{The five exceptional graphs for Theorem~\ref{theorem3}(1)(b), where $S=\{x_1,x_2\}$, $T=\{t_1,t_2,t_3,t_4,t_5\}$, and ``+'' represents the join of $H_i[S]$ and $H_i[T]$, $i\in [6,11]$.}
\label{f3}
\end{figure}
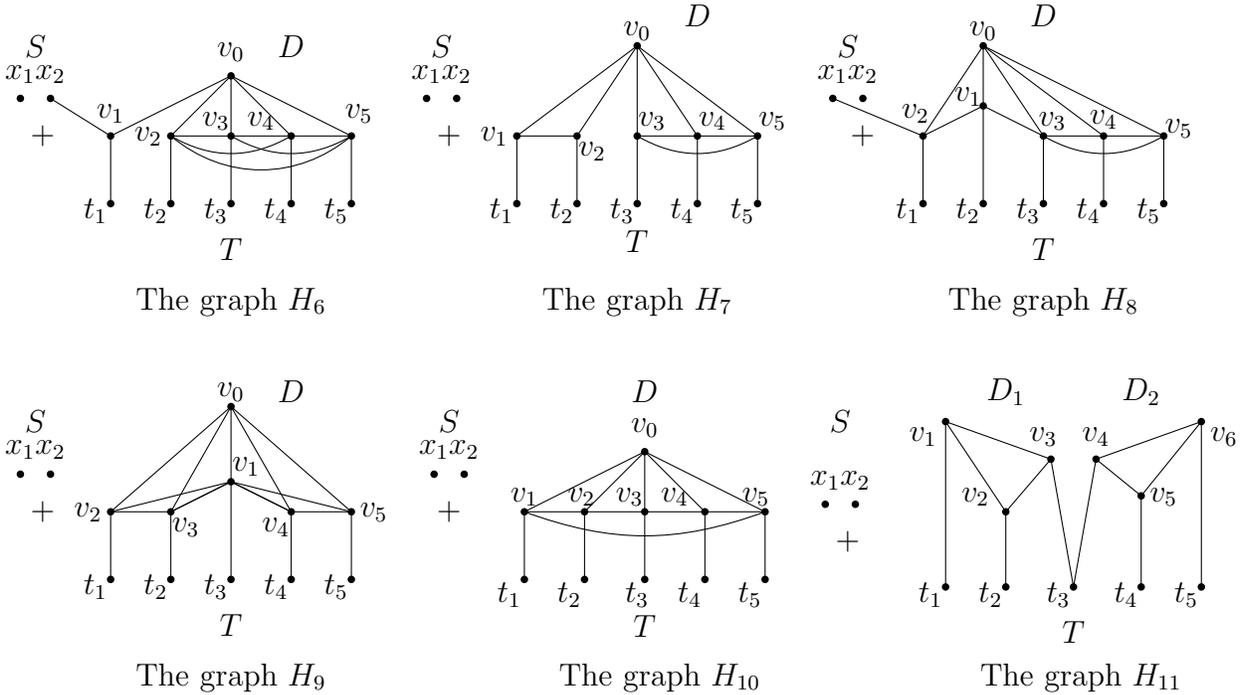


%
%
%
%

\begin{REM}[Examples demonstrating sharp toughness bounds]\label{remark} The toughness bounds in Theorems~\ref{theorem1} to~\ref{theorem3} are all sharp. 
\begin{enumerate}[(1)]
	\item  Theorem~\ref{theorem1}(1) when $R \in \{P_4 \cup P_1, P_3\cup 2P_1, P_2\cup 3P_1\}$  and $t=1$. The graph showing that the toughness 1 is best possible is the complete bipartite $K_{n-1,n}$ for any integer $n\ge 2$.  The  graph $K_{n,n-1}$ is $P_4$-free and so is $R$-free, with $\lim_{n \to \infty}\tau(K_{n,n-1})=\lim_{n \to \infty}\frac{n-1}{n}=1$, but contains no 2-factor.  
	
	\item  Theorem~\ref{theorem1}(2), Theorem~\ref{theorem2}(1)
	and  Theorem~\ref{theorem3}(1) and $t>1$. The graph showing  that the toughness  is best possible is the graph   $H_{12}$, which is constructed as below: let $p\ge 3$,  $K_p$ be a complete graph, and $y_1,y_2,y_3\in V(K_p)$ be distinct, $S=\{x\}$, and $T=\{t_1,t_2,t_3\}$, then $H_{12}$ is obtained from $K_p$, $S$
	and $T$ by adding edges $t_ix$ and  $t_iy_i$ for each $i\in [1,3]$. See Figure~\ref{f4}
	for a depiction. 
	By inspection, the graph is $5P_1$-free and $(P_4\cup 2P_1)$-free. So the graph is $R$-free for any $R\in \{5P_1, P_4\cup 2P_1, P_3\cup 3P_1, P_2\cup 4P_1, 6P_1, P_4\cup 3P_1, P_3\cup 4P_1, P_2\cup 5P_1\}$.  For any given $p \ge 3$, 
	the graph $H_{12}$ does not contain a 2-factor, as any 2-factor has to contain the edges $t_1x, t_2x$ and $t_3x$. We will show $\tau(H_{12})=1$
	in the last section. 
   \item For Theorem~\ref{theorem1}(3), Theorem~\ref{theorem2}(2)
   and  Theorem~\ref{theorem3}(3) and $t=\frac{3}{2}$:  note that 
   all the graphs $R$ in these cases contain $2K_2$ as an induced subgraph. Chv\'atal~\cite{chvatal} constructed a sequence $\{G_k\}_{k=1}^\infty$
   of split graphs (graphs whose vertex set can be partitioned into a clique and an independent set)  having no 2-factors and $\tau(G_k)=\frac{3k}{2k+1}$
    for each positive integer $k$. As the class of $2K_2$-free graphs 
    is a superclass of split graphs, $\frac{3}{2}$-tough is the best possible toughness bound for a $2K_2$-free graph to have a 2-factor.
    
   \item Theorem~\ref{theorem3}(2) and $t>\frac{7}{6}$.  The graph showing  that the toughness is best possible is the graph $H_5$ with $p\ge 6$, which is constructed as below: let $p\ge 5$,  $K_p$ be a complete graph, and $y_1,y_2,y_3, y_4,y_5\in V(K_p)$ be distinct, $S=\{x_1, x_2\}$, and $T=\{t_1,t_2,t_3, t_4,t_5\}$. Then $H_5$ is obtained from $K_p$, $S$
   and $T$ by adding edges $t_ix_j$ and  $t_iy_i$ for each $i\in [1,5]$ and each $j\in [1,2]$. 
    See Figure~\ref{f2} for a depiction. 
   By inspection, the graph is $7P_1$-free. For any given $p \ge 5$, 
   the graph $H_5$ does not contain a 2-factor, as any 2-factor has to contain at least three edges from one of $x_1$
   and $x_2$ to at least three vertices of $T$.   We will show $\tau(H_{5})=\frac{7}{6}$ when $p\ge 6$
   in the last section. 
   
\end{enumerate}
\end{REM}


\begin{figure}[!htb]
	\begin{center}
		
		\begin{tikzpicture}[scale=1]

\begin{scope}[scale=0.5, shift={(0,0)}]
\coordinate (S) at (-4,4);
\coordinate (D2) at (4,4);
\coordinate (T) at (0,0);
\coordinate (Tb) at (0,-0.5);
\node (d2) at (4,4.8) {$K_p$};
\node (t) at (0,-1) {$ T$};
\node (SL) at (-4.6, 4) {$x$};
\node (t) at (0,-2) {The graph $H_{12}$};

\coordinate (v4) at (3,3.25);
\coordinate (v5) at (4,3);
\coordinate (v6) at (5,3.25);
\coordinate (v7) at (-2,0.25);
\coordinate (v8) at (0,0.25);
\coordinate (v9) at (2,0.25);

\node[circle,fill=black,inner sep=1pt,label=$S$] at (S) {};

\node[circle,fill=black,inner sep=1pt, label=$t_1$] at (v7) {};
\node[circle,fill=black,inner sep=1pt, label=$t_2$] at (v8) {};
\node[circle,fill=black,inner sep=1pt, label=$t_3$] at (v9) {};
\node[circle,fill=black,inner sep=1pt, label=$y_1$] at (v4) {};
\node[circle,fill=black,inner sep=1pt, label=$y_2$] at (v5) {};
\node[circle,fill=black,inner sep=1pt, label=$y_3$] at (v6) {};

\draw[thick] (D2) ellipse (3 and 1.5);

\draw (S) to (v7) to (v4);
\draw (S) to (v8) to (v5);
\draw (S) to (v9) to (v6);
\end{scope}

	\end{tikzpicture}

\end{center}
\caption{Sharpness example for Theorem~\ref{theorem1}(2), Theorem~\ref{theorem2}(1)
	and  Theorem~\ref{theorem3}(1), where $S=\{x\}$ and $T=\{t_1,t_2, t_3\}$.}
\label{f4}
\end{figure}
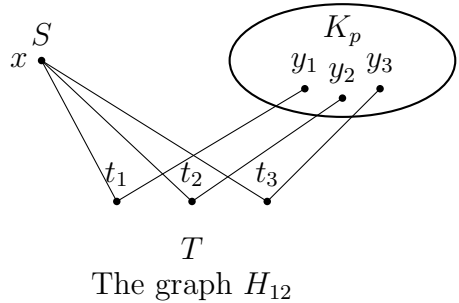		
To supplement Theorems~\ref{theorem1} to~\ref{theorem3},  
we show that the exceptional graphs in Figures~\ref{f1} to~\ref{f3}  satisfy the 
corresponding conditions below. 

\begin{THM}\label{theorem4a}
	The following statements hold. 
	\begin{enumerate}[(1)]
		\item The graph $H_i$ is $(P_2\cup 3P_1)$-free, contains no 2-factor, and $\tau(H_i)=1$ for each $i\in [0,4]$, the graph $H_1$ is also $(P_3\cup 2P_1)$-free. 
		  
		\item The graph $H_i$ is $(P_2\cup 5P_1)$-free and contains no 2-factor for each $i\in [5,11]$, $H_5$ with $p=5$ is   $(P_3 \cup 4P_1)$-free and $6P_1$-free. 
		Furthermore,  $\tau(H_5)=\frac{6}{5}$  when $p=5$ and  $\tau(H_i)= \frac{7}{6}$ for each $i\in [6,11]$.  
	\end{enumerate}
\end{THM}

We have explained that $H_5$ and $H_{12}$ are $R$-free for the 
corresponding linear forests $R$ and contain no 2-factor in 
Remark~\ref{remark}(2) and (4). The Theorem below 
is to verify the toughness of the graphs $H_{5}$  with $p\ge 6$ and $H_{12}$.

\begin{THM}\label{theorem4}
	The following statements hold. 
	\begin{enumerate}[(1)]
		\item $\tau(H_5)=\frac{7}{6}$ when $p\ge 6$; 
		\item $\tau(H_{12})=1$.
	\end{enumerate}
	
\end{THM}

The remainder of this paper is organized as follows. 
In section \ref{section:prelim}, we introduce more notation and preliminary results on proving existence of 2-factors in graphs. In section \ref{section:proof}, we prove Theorems~\ref{theorem1}-\ref{theorem3}. 
Theorems~\ref{theorem4a} and ~\ref{theorem4} are proved in  
the last section. 

\section{Preliminaries}\label{section:prelim}

One of the main proof ingredients of
Theorems~\ref{theorem1} to~\ref{theorem3} is to apply Tutte's
2-factor Theorem. We start with some notation. 
Let $S$ and $T$ be disjoint subsets of vertices of a graph $G$, and 
 $D$ be a component of $G-(S\cup T)$.
The component  $D$ is said to be an {\it odd component\/}
(resp.~{\it even component\/}) of $G-(S\cup T)$
if $e_G(D, T)\equiv 1\pmod{2}$
(resp.~$e_G(D, T)\equiv 0\pmod{2}$).
Let $h(S,T)$ be the number of all odd components of $G-(S\cup T)$. 
Define 
$$\delta(S, T)=2|S|-2|T|+\sum_{y\in T} d_{G-S}(y)-h(S,T).$$
It is easy to see that $\delta(S, T)\equiv 0\pmod{2}$
for every $S$,~$T\subseteq V(G)$
with $S\cap T=\emptyset$.
We use the following criterion for the existence of a $2$-factor,
which is a restricted form of Tutte's $f$-factor Theorem.

\begin{LEM}[Tutte~\cite{tutte}]\label{tutte's theorem}
	A graph $G$ has a $2$-factor if and only if
	$\delta(S, T)\ge 0$
	for every $S$,~$T\subseteq V(G)$
	with $S\cap T=\emptyset$.
\end{LEM}

An ordered pair $(S,T)$, consisting of
disjoint subsets of vertices $S$ and $T$ in a graph $G$,
is called a {\it barrier\/} if $\delta(S, T)\le -2$.
By Lemma~\ref{tutte's theorem},
if $G$ does not have a $2$-factor,
then $G$ has a barrier. In~\cite{EJC262}, a {\it biased barrier} of $G$ 
is defined as a barrier $(S, T)$ of $G$ such that among all the barriers of $G$,
	\begin{enumerate}[(1)]
		\item $|S|$ is maximum; and
		\item subject to (1), $|T|$ is minimum.
	\end{enumerate}

The following properties of a biased barrier were derived in~\cite{EJC262}. 

\begin{LEM}\label{lem:biasedbarrier}
	Let $G$ be a graph without  a $2$-factor,
	and let $(S, T)$ be a biased barrier of $G$.
	Then each of the following holds.
	\begin{enumerate}[(1)]
		\item The set
		$T$ is  independent  in $G$.
		\item If $D$ is an even component
		with respect to $(S, T)$,
		then $e_G(T, D)=0$.
		\item If $D$ is an odd component
		with respect to $(S,T)$,
		then for any  $y\in T$, $e_G(y, D) \le 1$.
		\item If $D$ is an odd component
		with respect to $(S,T)$,  then for any $x\in V(D)$,
		$e_G(x,T)\le 1$.
		\end{enumerate}
\end{LEM}

Let $G$ be a graph without a 2-factor and $(S,T)$ be a barrier of $G$. 
For an integer $k\ge 0$, 
let $\mathcal{C}_{2k+1}$ denote the set of odd components $D$ of $G - (S\cup T)$ 
such that $e_G(D,T)=2k+1$. 
The following result was proved as a claim in~\cite{EJC262} but we include a short proof 
here for self-completeness.   
\begin{LEM}\label{LEM:T>S}
	Let $G$ be a graph without  a $2$-factor,
	and let $(S, T)$ be a biased barrier of $G$. Then 
	$|T| \ge  |S|+\sum_{k\ge 1}k|\CC_{2k+1}|+1$. 
\end{LEM}
\pf Let $U=V(G)\setminus S$. 
Since $(S, T)$ is a barrier,
\[
\begin{split}
\delta(S, T)&=2|S|-2|T|+\sum_{y\in T}d_{G-S}(y)-h(S,T)\\
&=2|S|-2|T|+\sum_{y\in T}d_{G-S}(y)-\sum_{k\ge 0} |\CC_{2k+1}| \le -2.
\end{split}
\]
By Lemma~\ref{lem:biasedbarrier}(1) and Lemma~\ref{lem:biasedbarrier}(2),
\[
\sum_{y\in T}d_{G-S}(y)=\sum_{y\in T}e_G(y, U)
=e_G(T, U)=\sum_{k\ge 0}(2k+1)|\CC_{2k+1}|.
\]
Therefore,
we have
\[
-2\ge  2|S|-2|T|+\sum_{k\ge 0}(2k+1)|\CC_{2k+1}|-\sum_{k\ge 0}|\CC_{2k+1}|,
\]
which yields
$|T| \ge  |S|+\sum_{k\ge 1}k|\CC_{2k+1}|+1$.
\qed 


We use the following lemmas in our proof.

\begin{LEM} \label{lem:NoC1s}
	Let $t\ge 1$, $G$ be a $t$-tough graph on at least three vertices containing no $2$-factor,
	and  $(S,T)$ be a  barrier of $G$. Then the following statements hold. 
	\begin{enumerate}[(1)]
		\item If	 $\mathcal{C}_1 \ne \emptyset$, then $|S|+1\ge 2t$.  
		Consequently, $S=\emptyset$ implies $\mathcal{C}_1 = \emptyset$,  
		and   $|S|=1$ implies $\mathcal{C}_1 = \emptyset$ when $t>1$. 
		\item  $\bigcup_{k\geq 1}\mathcal{C}_{2k+1} \neq \emptyset$. 
	\end{enumerate}
\end{LEM} 
\pf	Since $G$
	is $1$-tough and thus is 2-connected, $d_G(y) \ge 2$
	for every $y\in T$.  This together with Lemma~\ref{lem:biasedbarrier}(1)-(3) 
	implies 
 $|S| + \sum_{k\ge 0}|\mathcal{C}_{2k+1}| \ge 2$. 
	
For the first part of (1), 	suppose to the contrary  that $|S|+1<2t$. Let $D\in \mathcal{C}_1$ and  $y\in V(T)$ be adjacent in $G$ to some vertex $v\in V(D)$.  As $e_G(D,T)=e_G(D,y)=1$, 
$|S| + \sum_{k\ge 0}|\mathcal{C}_{2k+1}| \ge 2$.  and $|T| \ge |S|+1$ by Lemma~\ref{LEM:T>S}, 
we have $c(G-(S\cup\{y\}))\ge 2$ regardless of whether or not $S=\emptyset$. But $c(G-(S\cup\{y\}))\ge 2$ implies $\tau(G) < 2t/2=t$, 
contradicting $G$ being $t$-tough.  The second part of (1) is a consequence of the first part.

	For (2), suppose to the contrary that 
	$\bigcup_{k\geq 1}\mathcal{C}_{2k+1} = \emptyset$. 
	By Lemma~\ref{lem:NoC1s}(1), $|S| + |\mathcal{C}_{1}| \ge 2$ implies $|S| \ge 1$. 
	Consequently,  $|T| \ge 2$ by Lemma~\ref{LEM:T>S}.  As every component of $G-(S\cup T)$
	in $\CC_1$ is connected to exactly one vertex of $T$, $S$
	is a cutset of $G$ with $c(G-S)\ge |T|$. However, 
	$|T|  \geq |S| + \sum_{k\geq 1}k|\mathcal{C}_{2k+1}|+1=|S|+1$, 
	implying $\tau(G)<1$, a contradiction. 
\qed

A path $P$ connecting two vertices $u$ and $v$ is called 
a {\it $(u,v)$-path}, and we write $uPv$ or $vPu$ in  order to specify the two endvertices of 
$P$. Let $uPv$ and $xQy$ be two disjoint paths. If $vx$ is an edge, 
we write $uPvxQy$ as
the concatenation of $P$ and $Q$ through the edge $vx$. Let $G$ be a graph without  a $2$-factor,
and let $(S, T)$ be a  barrier of $G$.
For $y\in T$, define 
$$
h(y)=|\{ D: \text{ $D\in \bigcup_{k\ge 1} \CC_{2k+1}$} \quad \text{and} \quad e_G(y,D) \ge 1\}|. 
$$

\begin{LEM} \label{lem:P4}
Let $G$ be a graph without  a $2$-factor,
and let $(S, T)$ be a biased barrier of $G$. Then the following statements hold.
\begin{enumerate}[(1)]
	\item If $|\bigcup_{k\geq 1}\mathcal{C}_{2k+1} | \ge 1$, then $G$ contains an induced $P_4 \cup aP_1$, where $a=|T|-2$. 

	\item If there exists $y_0\in T$ with $h(y_0) \ge 2$, then for some integer $b\ge 7$, $G$ contains an induced $P_b \cup aP_1$, where $a=|T|-3$. Furthermore,  an induced $P_b \cup aP_1$ can be taken such that the vertices in $aP_1$
	are from $T$ and the path $P_b$ has the form $y_1x_1^*P_1x_1y_0x_2P_2x_2^*y_2$, where 
	$y_0,y_1,y_2\in T$ and $x_1^*P_1x_1$ and $x_2^*P_2x_2$ are respectively contained in two distinct 
	components from  $\bigcup_{k\geq 1}\mathcal{C}_{2k+1}$ such that $e_G(x,T)=0$
	for every internal vertex $x$ from $P_1$ and $P_2$. 
\end{enumerate} 
\end{LEM}
\pf
	Lemma~\ref{lem:biasedbarrier}(1), (3) and (4) will be applied frequently in the arguments sometimes without mentioning it. 

Let $D\in \bigcup_{k\geq 1}\mathcal{C}_{2k+1}$.  The existence of $D$ implies $|T| \ge 3$ and $|V(D)| \ge 3$ by Lemma~\ref{lem:biasedbarrier}(3) and (4). We claim that
for a fixed vertex $x_1\in V(D)$ such that $e_G(x_1,T)=1$, 
 there exists distinct $x_2\in V(D)$
and an induced $(x_1,x_2)$-path $P$ in $D$  with the following two properties:
(a) $e_G(x_2,T)= 1$,  and (b) $e_G(x,T)=0$
for every $x\in V(P)\setminus \{x_1,x_2\}$. Note that the vertex $x_1$ exists by Lemma~\ref{lem:biasedbarrier}(4). Let $y_1 \in T$ be the vertex 
such that  $e_G(x_1,T)=e_G(x_1,y_1)=1$ and $W=N_G(T\setminus \{y_1\})\cap V(D)$. 
By Lemma~\ref{lem:biasedbarrier}(4), $x_1\not \in W$. 
Now in $D$, we find a shortest path $P$ connecting $x_1$
and some vertex from $W$, say $x_2$.  Then $x_2$ and $P$ satisfy properties (a)
and (b), respectively. 
 Let $y_2\in T$ such that 
$e_G(x_2,T)=e_G(x_2,y_2)=1$.  The vertex $y_2$ uniquely exists by the choice  $x_2$ 
and Lemma~\ref{lem:biasedbarrier}(4). 
 By Lemma~\ref{lem:biasedbarrier}(1) and (4),  and the choice of $P$,   we know that $y_1x_1Px_2y_2$ and $T\setminus\{y_1,y_2\}$ together contains an induced $P_4\cup aP_1$. This  proves (1). 

We now prove (2). 
By Lemma~\ref{lem:biasedbarrier}(3), the existence of $y_0$ implies $|\bigcup_{k\geq 1}\mathcal{C}_{2k+1}| \ge 2$, which in turn gives $|T| \ge 3$ by Lemma~\ref{lem:biasedbarrier}(3) again. We let $D_1,D_2\in \bigcup_{k\geq 1}\mathcal{C}_{2k+1}$ be distinct such that 
$e_G(y_0,D_1)=1$ and $e_G(y_0,D_2)=1$.  Let $x_i\in D_i$ such that 
$e_G(y_0,D_i)=e_G(y_0,x_i)=1$. 
By the argument in the first paragraph above, 
we can find $x^*_i\in V(D_i)\setminus \{x_i\}$ and an $(x_i,x_i^*)$-path $P_i$ in $D_i$ for each $i\in \{1,2\}$. By the choice of $P_i$ and 
Lemma~\ref{lem:biasedbarrier}(4), there are unique $y_1,y_2\in T\setminus\{y_0\}$ such that $x_i^*y_i\in E(G)$. If $y_1\ne y_2$, 
by the choice of $P_1$
and $P_2$ and Lemma~\ref{lem:biasedbarrier}(1) and (4), we know that $y_1x_1^*P_1x_1y_0x_2P_2x_2^*y_2$ and $T\setminus\{y_0,y_1,y_2\}$ together contain an induced $P_b\cup aP_1$  for some integer $b\ge 7$.  
Thus we assume $y_1=y_2$.  Then the vertex $y_1$ can also play the role of $y_0$. 
Let  $W=N_G(T\setminus \{y_0,y_1\})\cap V(D_2)$. 
 By Lemma~\ref{lem:biasedbarrier}(4), $x_2, x_2^*\not \in W$. 
 Now in $D_2$, we find a shortest path $P_2^*$ connecting some vertex of $\{x_2,x_2^*\}$
 and some vertex from $W$, say $z$.  If $P_2^*$
 is an $(x_2,z)$-path, then $y_1x^*_1P_1x_1y_0x_2P_2^*z$ and $T\setminus\{y_0,y_1,y_2\}$ together
 contain an induced $P_b\cup aP_1$. 
  If $P_2^*$
 is an $(x^*_2,z)$-path, then $y_0x_1P_1x^*_1y_1x_2^*P_2^*z$ and $T\setminus\{y_0,y_1,y_2\}$ together
 contain an induced $P_b\cup aP_1$.
  The second part of 
(2) is clear by the construction above. 
\qed

Let $G$ be a non-complete graph. A cutset $S$
of $V(G)$ is a \emph{toughset} of $G$ if  $\frac{|S|}{c(G-S)}=\tau(G)$.  
\begin{LEM}\label{lem:tough-set}
If $G$ is a connected  graph and $S$ is a toughset of $G$, then for every $x\in S$, $x$ is adjacent in $G$ to vertices from at least two components of 
$G-S$. 
\end{LEM}
\pf Assume to the contrary that there exists $x\in S$
such that $x$ is adjacent in $G$ to vertices from at most one 
component of $G-S$. Then $c(G-(S\setminus\{x\}))=c(G-S)\ge 2$
and 
$$\frac{|S\setminus\{x\}|}{c(G-(S\setminus\{x\}))}<\frac{|S|}{c(G-S)}=\tau(G),$$
contradicting $G$ being $\tau(G)$-tough.  
\qed

\section{Proof of Theorems~\ref{theorem1}, \ref{theorem2}, and~\ref{theorem3}}\label{section:proof}

Let $R$ be any linear forest on at most 7 vertices. 
If $G$ is $R$-free, then $G$
 is also $R^*$-free for any supergraph $R^*$ of $R$. To prove Theorems~\ref{theorem1} to~\ref{theorem3}, 
 we will show that the corresponding statements hold for a supergraph $R^*$  
 of $R$, which simplifies the cases of   possibilities of  $R$.   Let us first 
 list the supergraphs that we will use. 
\setlength{\parindent}{0pt} 

\begin{enumerate}[(1)]
	\item $P_4\cup 3P_1$ is a supergraph of the following graphs: 
		$P_4 \cup 2P_1, P_3 \cup 3P_1,$ and $P_2 \cup 4P_1$;
		\item $6P_1$ is a supergraph of $5P_1$; 
		\item $P_3\cup 2P_2$ is a supergraph of $3P_2$; 
	\item $P_7 \cup 2P_1$  is a supergraph of the following graphs:
	\begin{enumerate}
		\item $P_5, P_3\cup P_2, 2P_2\cup P_1$;
		\item $P_6, P_5\cup P_1, P_4\cup P_2, 2P_3, P_3\cup P_2\cup P_1,  2P_2\cup 2P_1$;
		\item $P_7, P_6\cup P_1,  P_5\cup 2P_1,  P_4\cup P_2\cup P_1, 2P_3\cup P_1,  P_3\cup P_2\cup 2P_1,  2P_2\cup 3P_1$. 
	\end{enumerate}
\end{enumerate}

Those supergraphs above together with the graphs $R$ listed below  cover all 
the 33  $R$ graphs described in Theorems~\ref{theorem1} to~\ref{theorem3}. 
Theorems~\ref{theorem1} to~\ref{theorem3} are then consequences of the theorem below. 
\begin{THM}\label{thm9}
	Let $t>0$ be a real number, $R$ be a linear forest, and $G$ be a $t$-tough $R$-free graph on at least 3 vertices. Then $G$ has a 2-factor provided that 
	\begin{enumerate}[(1)]
	    	\item $R \in \{P_4 \cup P_1, P_3\cup 2P_1, P_2\cup 3P_1\}$  and $t=1$ unless
	    \begin{enumerate}
	    	\item $R = P_2 \cup 3P_1$, and $G \cong H_0$ or $G$ contains $H_1$, $H_2$, $H_3$ or $H_4$ as a spanning subgraph such that $E(G)\setminus E(H_i) \subseteq E_G(S, V(G)\setminus (T\cup S))$ for each $i\in [1,3]$, 
	    	 where $H_i$, $S$ and $T$ are defined in Figure~\ref{f1}.    
	    	\item $R=P_3 \cup 2P_1$ and $G$ contains $H_1$ as a spanning subgraph such that $E(G)\setminus E(H_1) \subseteq E_G(S, V(G)\setminus (T\cup S))$.  
	    \end{enumerate}
	    \item $R \in \{P_4\cup 3P_1, P_3\cup 4P_1, P_2\cup 5P_1, 6P_1\}$ and $t>1$ unless 
	    \begin{enumerate}
	    \item  when $R\ne P_4\cup 3P_1$, $G$ contains $H_5$ with $p=5$ as a spanning subgraph such that $E(G)\setminus E(H_5) \subseteq E_G(S, V(G)\setminus (T\cup S)) \cup E(G[S])$, where $H_5$, $S$ and $T$ are defined in Figure~\ref{f2}.  
	    \item $R=P_2 \cup 5P_1$ and $G$ contains one of  $H_6, \ldots, H_{11}$ as a spanning subgraph such that $E(G)\setminus E(H_i) \subseteq E_G(S, V(G)\setminus (T\cup S)) \cup E(G[S]) \cup  E(G[ V(G)\setminus (T\cup S)])$,    where $H_i$, $S$ and $T$ are defined in Figure~\ref{f3} for each $i\in [6,11]$. 
	    \end{enumerate}
	    	    \item $R=7P_1$ and $t>\frac{7}{6}$ unless $G$ contains $H_5$ with $p=5$ as a spanning subgraph such that $E(G)\setminus E(H_5) \subseteq E_G(S, V(G)\setminus (T\cup S)) \cup E(G[S])$. 
	    \item $R \in \{P_7\cup 2P_1, P_5 \cup P_2, P_4 \cup P_3, P_3 \cup 2P_2, 3P_2 \cup P_1 \}$ and $t=3/2$.
	   \end{enumerate} \end{THM}

\pf Assume by contradiction that $G$ does not have a 2-factor. By Lemma~\ref{tutte's theorem}, $G$ has a barrier. We choose $(S,T)$ to be a biased barrier.  Thus $(S,T)$ and $G$ satisfy all the properties listed in Lemma~\ref{lem:biasedbarrier}. These properties will be used frequently 
even without further mentioning sometimes.  
By Lemma~\ref{LEM:T>S},
\begin{equation}\label{eqn1}
|T| \ge   |S|+\sum_{k\ge 1}k|\CC_{2k+1}|+1. 
\end{equation}
Since $t\ge 1$, by Lemma~\ref{lem:NoC1s}(2), we know that 
\begin{equation}\label{eqn2}
\bigcup_{k\geq 1}\mathcal{C}_{2k+1} \neq \emptyset. 
\end{equation}
This implies  $|T| \ge 3$ and so $G$ contains an induced $P_4\cup P_1$ by Lemma~\ref{lem:P4} (1). 
Thus we assume $R\ne P_4\cup P_1$ in the rest of the proof. 

\begin{CLAIM}\label{clm0}
$R\not\in \{P_3 \cup 2P_1, P_2 \cup 3P_1\}$ unless $G$ falls under one of the exceptional cases as in (a) and (b) of Theorem~\ref{thm9}(1). 
\end{CLAIM}
\pf Assume instead that $R \in \{P_3 \cup 2P_1, P_2 \cup 3P_1\}$. Thus $t=1$.  We may assume that 
$G$ does not fall under any of the exceptional cases as in (a) and (b) of Theorem~\ref{thm9} (1). 

It must be the case that $|T|=3$, as otherwise  $G$ contains an induced $P_4\cup 2P_1$ by Lemma~\ref{lem:P4}(1),
and so contains an induced $R$. 
 By Equation~\eqref{eqn1}, we have $|\bigcup_{k\geq 1}\mathcal{C}_{2k+1} |+|S| \le 2$. By Lemma~\ref{lem:NoC1s}(1), we have that 
 $\CC_0=\emptyset$ if $S=\emptyset$. 
 Since $G$ is 1-tough and so $\delta(G) \ge 2$,  Lemma~\ref{lem:biasedbarrier}(1)-(3) implies that $|\bigcup_{k\geq 1}\mathcal{C}_{2k+1} |+|S| = 2$. By~\eqref{eqn2}, we have the two cases below. 
 
\medskip 
{\normalsize \scshape Case 1:} $|\bigcup_{k\geq 1}\mathcal{C}_{2k+1} | = 2$ and $S= \emptyset$.  \\

Let $D_1, D_2 \in \bigcup_{k\geq 1}\mathcal{C}_{2k+1}$ be the two odd components of $G-(S\cup T)$. 
Since $|T|=3$, Lemma~\ref{lem:biasedbarrier}(3) implies that $e_G(D_i,T)=3$
for each $i\in [1,2]$. Let $y\in T$ and $x\in V(D_1)$ such that $xy\in E(G)$. 
We let $x_1$ be a neighbor of $x$ from $D_1$. Then $yxx_1$
is an induced $P_3$ by Lemma~\ref{lem:biasedbarrier}(3). 
Let $y_1 \in T\setminus\{y\}$ such that $y_1 x_1\not\in E(G)$, 
which is possible as $|T|=3$ and $e_G(x_1,T) \le 1$ by Lemma~\ref{lem:biasedbarrier}(4).
We now let $x_2 \in V(D_2)$ such that $e_G(x_2, \{y,y_1\})=0$, 
which is again possible as $|N_G(T)\cap V(D_2)|=3$ and each vertex of $D_2$
is adjacent in $G$ to at most one vertex of $T$.  However, $yxx_1, y_1$ and $ x_2$ 
together form an induced copy of $P_3 \cup 2P_1$. 
Therefore, we assume $R = P_2 \cup 3P_1$. 

We first claim that $|V(D_i)|=3$ for each
$i\in [1,2]$. Otherwise, say $|V(D_2)| \ge 4$. 
Let $y\in T$ and $x\in V(D_1)$ such that $xy\in E(G)$. 
Take $x_1 \in V(D_2)$ such that $e_G(x_1, T)=0$, which exists as $|N_G(T)\cap V(D_2)|=3$. 
Then $xy, x_1$ and $T\setminus \{y\}$ together form an induced copy of $P_2 \cup 3P_1$, giving a contradiction. 
We next claim that $D_i=K_3$ for each
$i\in [1,2]$. Otherwise, say $D_1\ne K_3$. As $D_1$ is connected, it follows that $D_1=P_3$. 
If also $D_2\ne K_3$ and so $D_2=P_3$, then deleting the two vertices of degree 2 from both $D_1$
and $D_2$ gives three components (note that each vertex of $T$ is adjacent in $G$ to one vertex of $D_1$ and one vertex of $D_2$), showing that $\tau(G) \le 2/3 <1$. 
Thus $D_2= K_3$. We let $x_1, x_2\in V(D_1)$ be nonadjacent, $y_1,y_2\in T$
such that $e_G(x_i,y_i)=1$ for each $i\in [1,2]$, and $z_1,z_2\in V(D_2)$
such that $e_G(y_i,z_i)=1$ for each $i\in [1,2]$. 
Let $y\in T\setminus \{y_1,y_2\}$. Then $z_1z_2, y, x_1$ and $x_2$ together form an induced copy of $P_2 \cup 3P_1$, giving a contradiction.  

Thus $|V(D_i)|=3$ and $D_i=K_3$ for each
$i\in [1,2]$. However, this implies that $G\cong H_0$. 


\medskip 
{\normalsize \scshape Case 2:} $|\bigcup_{k\geq 1}\mathcal{C}_{2k+1} | = 1$ and $|S|=1$.  \\

Let $D\in \bigcup_{k\geq 1}\mathcal{C}_{2k+1}$ be the  odd component of $G-(S\cup T)$. 
Assume first that $R=P_3\cup2P_1$. Then we have $|V(D)|=3$. 
Otherwise, $|V(D)| \ge 4$. Let $x\in V(D)$ such that $e_G(x,T)=0$ 
and $P$ be a shortest path of $D$ from $x$
to a vertex, say 
 $x_1\in V(D)\cap N_G(T)$. Let $y\in T$ such that  $e_G(x_1,y)=1$. 
 Then $xPx_1y$ and $T\setminus\{y\}$ form an induced copy of $R$, a contradiction. 

Since $G$ does not contain $H_1$ as a spanning subgraph such that $E(G)\setminus E(H_1) \subseteq E_G(S, V(G)\setminus (T\cup S))$, it follows that $D\ne K_3$.  As $D$ is connected, it follows that $D=P_3$. 
Now deleting the vertex in $S$
together with the degree 2 vertex of $D$ produces three components, showing that $\tau(G) \le 2/3 <1$. 

Therefore, we assume now that $R=P_2\cup3P_1$. Since $G$ does not contain $H_1$ as a spanning subgraph such that $E(G)\setminus E(H_1) \subseteq E_G(S, V(G)\setminus (T\cup S))$, the argument for the 
case $R=P_3\cup2P_1$ above implies that $|V(D)| \ge 4$. 
We claim that $|V(D)| = 4$. If $|V(D)| \ge 5$, we let $x_1,x_2\in V(D) \setminus N_G(T)$ be any two distinct 
vertices. If $x_1x_2\in E(G)$, then $x_1x_2$ together with $T$ form an induced copy of $R$, a contradiction. 
Thus $V(D) \setminus N_G(T)$ is an independent set in $G$. 
However, $c(G-(S\cup (N_G(T)\cap V(D))))=|T|+|V(D) \setminus N_G(T)| \ge 5$, implying that $\tau(G) \le 4/5<1$. 

Thus $|V(D)|=4$. Let $x\in V(D)$ such that $e_G(x,T)=0$. Since $G$ does not contain  $H_i$  as a spanning subgraph such that $E(G)\setminus E(H_i) \subseteq E_G(S, V(G)\setminus (T\cup S))$ for each $i\in [2,4]$,  it follows that  either $d_D(x) \le 2$ or $d_D(x)=3$ and $D=K_{1,3}$. 
 If $d_D(x)=3$, then as $D=K_{1,3}$, we have $c(G-(S\cup\{x\}))=3$, implying $\tau(G) \le 2/3<1$. 
 Thus $d_D(x) \le 2$. Let $V(D)=\{x, x_1, x_2, x_3\}$ and assume $xx_1\not\in E(D)$. 
 Then $c(G-(S\cup\{x_2, x_3\}))=4$, implying $\tau(G) \le 3/4<1$. The proof of Case 2 is complete. 
  \qed 

Thus by Claim~\ref{clm0} and the fact that $R \ne P_4\cup P_1$, we can assume $R\not\in \{P_4\cup P_1, P_3 \cup 2P_1, P_2 \cup 3P_1\}$ 
from this point on.
Therefore we have  $t>1$. This implies that  $G$ is 3-connected and so $\delta(G) \ge 3$. Thus
$|S|+ |\bigcup_{k\ge 0}\CC_{2k+1}| \ge 3$ by Lemma~\ref{lem:biasedbarrier}(1)-(4). 

\begin{CLAIM}\label{clm1}
$|T| \ge 5$. 
\end{CLAIM}
\pf 
Equation~\eqref{eqn2} implies  $|T|\ge 3$. Assume to the contrary that $|T| \le 4$.
We consider the following two cases. 

\medskip 
{\normalsize \scshape Case 1:} $|T|=3$.  \\

Since $|S|+ |\bigcup_{k\ge 0}\CC_{2k+1}| \ge 3$, we already have a contradiction to Equation~\eqref{eqn1} if $\CC_1=\emptyset$. Thus $\CC_1 \ne \emptyset$, which gives 
$|S|\ge 2$ by Lemma~\ref{lem:NoC1s}(1). However, we again get a contradiction to Equation~\eqref{eqn1} as  $\bigcup_{k\geq 1}\mathcal{C}_{2k+1} \neq \emptyset$ by Equation~\eqref{eqn2}. \\

{\normalsize \scshape Case 2:} $|T|=4$.\\

By Lemma~\ref{lem:biasedbarrier} (3), we know that $\CC_{2k+1} =\emptyset$
for any $k\ge 2$. 
First assume $|S| \le 1$.  Then $\CC_1=\emptyset$ by Lemma~\ref{lem:NoC1s} (1). 
By Lemma~\ref{lem:biasedbarrier}, there are at least $3|T|=12$
edges going from $T$ to vertices in $S$
and components in $\bigcup_{k\geq 1}\mathcal{C}_{2k+1}$. 
As $\CC_{2k+1} =\emptyset$
for any $k\ge 2$, it follows that $|\CC_3|\ge 4$  if $|S|=0$
and $|\CC_3|\ge 3$ if $|S|=1$, contradicting Equation~\eqref{eqn1}. 

Next, assume $|S| \ge 2$. 
By Equations~\eqref{eqn1} and~\eqref{eqn2}, we have $|S|=2$. 
 Let $D$ be the single component in $\mathcal{C}_{3}$.  Define $W_D$ to be a set of $2$ vertices in $D$ which are all adjacent in $G$ to some vertex from $T$. Then $S\cup W_D$ is a cutset in $G$ such that $|S\cup W_D|=4$ and $c(G - (S\cup W_D))\ge |T|=4$,
 contradicting $\tau(G)\ge t>1$. 
\qed
\\

By Claim~\ref{clm1} and Lemma~\ref{lem:P4} (1), we 
see that $G$ contains an induced $R=P_4\cup 3P_1$. Thus we may assume $R\not\in \{P_4\cup P_1, P_3 \cup 2P_1, P_2 \cup 3P_1, P_4\cup 3P_1\}$ 
from this point on.

%

\begin{CLAIM}\label{clm5}
$R\not\in \{ P_3\cup 4P_1, P_2\cup 5P_1, 6P_1, 7P_1\}$ unless  $G$ falls under the exceptional cases as in (a) and (b) of Theorem~\ref{thm9}(2). 
\end{CLAIM}
\begin{proof}
	We may assume that $G$ does not fall under the exceptional cases as in (a) and (b) of Theorem~\ref{thm9}(2). 
	Thus we show that $R\not\in \{ P_3\cup 4P_1, P_2\cup 5P_1, 6P_1, 7P_1\}$.

Assume to the contrary that $R\in \{ P_3\cup 4P_1, P_2\cup 5P_1, 6P_1, 7P_1\}$. 
By Lemma~\ref{lem:P4}(1), $G$ contains an induced $P_4\cup aP_1$, where $a=|T|-2$. 
If $a\ge 5$, then each of $P_3\cup 4P_1, P_2\cup 5P_1, 6P_1$,  and $7P_1$ is 
an induced subgraph of $P_4\cup aP_1$, a contradiction. Thus $a \le 4$
and so $|T| \le 6$. 
 As $|T| \le 6$, we have that $\bigcup_{k>2} \mathcal{C}_{2k+1} = \emptyset$ by Lemma~\ref{lem:biasedbarrier} (3). 
Since $G$ is more than $1$-tough and so is 3-connected, we have $\delta(G) \ge 3$.  
 By Claim \ref{clm1}, $|T|\geq 5$. Thus, we have two cases.

\medskip 
{\normalsize \scshape Case 1:} $|T|=5$.  
 
 As $|T|=5$,  we have $ \mathcal{C}_{2k+1}=\emptyset$ for any $k\ge 3$. 
 We consider two cases regarding whether or not $|\CC_3 \cup \CC_5| \ge 2$. 
 
 {\normalsize \scshape Case 1.1:} $|\CC_3 \cup \CC_5| =1$. 
 
 Let $D\in \CC_{2k+1} \subseteq \CC_3 \cup \CC_5$.  
   By Equation (\ref{eqn1}), $5 \geq |S| + k + 1$, so $|S|\leq 4-k$.  
  If $k=1$, 
 let $W_D$ be a set of $2k$ vertices (which exist by Lemma~\ref{lem:biasedbarrier} (4)) from $D$ which are adjacent in $G$ to vertices from $T$. Then $S \cup W_D$ forms a cutset and we have 
 \begin{align*}
 	t &\leq \frac{|S|+2k}{5} \le \frac{4+k}{5} = \frac{5}{5}=1, 
 \end{align*}
 contradicting $t>1$.  
 Thus we assume $k=2$.  We consider two subcases. 
 
 \medskip 
 {\normalsize \scshape Case 1.1.1:} $|V(D)|\ge 6$. 

 For $R=P_3 \cup 4P_1$, 
 let $x\in V(D)$  such that $e_G(x,T)=0$. 
 Let $P$ be a shortest path in $D$ from $x$
 to a vertex, say $x^*$ from $N_G(T)\cap V(D)$. 
 Let $y^*\in T$ such that $e_G(x^*,y^*)=1$. Then $xPx^*y^*$
 and $T\setminus \{y^*\}$ contain $P_3\cup 4P_1$ as an induced subgraph. 
 We consider next that $R=6P_1$. 
  Then $T$ and the vertex of $D$ that is not adjacent in $G$ to any vertex from $T$
 for an induced $6P_1$, giving a contradiction. 
  For $R=7P_1$, 
  let $W_D$ be the set of $2k+1$ vertices (which exist by Lemma~\ref{lem:biasedbarrier}(4)) from $D$ which are adjacent in $G$ to vertices from $T$. Then $S \cup W_D$ forms a cutset and we have 
 \begin{align*}
 	t &\leq \frac{|S|+2k+1}{|T|+1} \le \frac{4+k+1}{6} = \frac{7}{6},  
 \end{align*}
 giving a contradiction to $t> 7/6$.

 Lastly,   we consider $R=P_2\cup 5P_1$.  
For any $x\in V(D)$ such that $e_G(x,T)=0$, it must be the case that $x$
is adjacent in $G$ to every vertex from $N_G(T)\cap V(D)$.
Otherwise, let $x^*\in N_G(T)\cap V(D)$ such that $xx^*\not\in E(G)$. 
Let $y^*\in T$ such that $e_G(x^*,y^*)=1$. Then $x^*y^*$
and $(T\setminus\{y^*\}) \cup \{x\}$ contain $P_2\cup 5P_1$ as an induced subgraph.
Furthermore, if $|V(D)|-|N_G(T)\cap V(D)| \ge 2$, then $V(D)\setminus (N_G(T)\cap V(D))$
is an independent set in $G$. Otherwise, an edge with both endvertices from $V(D)\setminus (N_G(T)\cap V(D))$
together with $T$ induces $P_2\cup 5P_1$. 
Thus if $|V(D)| \ge 7$, let $W_D$ be the set of $2k+1$ vertices (which exist by Lemma~\ref{lem:biasedbarrier}(4)) from $D$ which are adjacent in $G$ to vertices from $T$. Then $S \cup W_D$ forms a cutset and we have 
\begin{align*}
	t &\leq \frac{|S|+5}{|T|+2} \le \frac{7}{7},  
\end{align*}
giving a contradiction to $t> 1$.  Thus $|V(D)|=6$. 
Let $x\in V(D)$  be the vertex such that $e_G(x,T)=0$.
Then it must be the case that $D-x$ has at most two components. 
Otherwise, we have $t\le \frac{|S\cup \{x\}|}{3}=1$.

Assume first that $c(D-x)=2$.  Let $D_1$ and $D_2$ be the two components of 
$D-x$, and assume further that $|V(D_1)| \le |V(D_2)|$.  Then as $|V(D-x)|=5$, we have two possibilities: 
 either $|V(D_1)|=1$ and $|V(D_2)|=4$ or $|V(D_1)|=2$ and $|V(D_2)|=3$.  Since $\delta(G) \ge 3$, if $|V(D_1)|=1$, then 
 the vertex from $D_1$ must be adjacent in $G$ to at least one vertex from $S$. 
 When 
 $|V(D_2)|=4$ and $D_2 \ne K_4$, then $D_2$ has a cutset $W$ of size 2 such  that $c(D_2-W)=2$. 
 Then $S\cup W\cup \{x\}$ is a cutset of $G$ such that $c(G-(S\cup W\cup \{x\}))=5$, showing that $t\le 1$. 
 Thus $D_2=K_4$.  However, this shows that $G$ contains $H_6$ as a spanning subgraph. 
When  $|V(D_2)|=3$ and $D_2 \ne K_3$, then $D_2$ has a cutvertex $x^*$.
Then $S\cup  \{x, x^*\}$ is a cutset of $G$ such that $c(G-(S\cup \{x, x^*\}))=4$, showing that $t\le \frac{4}{4}=1$. 
 Thus $D_2=K_3$; however, this shows that $G$ contains $H_7$ as a spanning subgraph. 

Assume then that $c(D-x)=1$. Let $D^*=D-x$. If $\delta(D^*) \ge 3$, then $D^*$ is Hamiltonian and so $G$ contains $H_{10}$ as a spanning subgraph. 
Thus we assume $\delta(D^*) \le 2$.  

Assume first that $D^*$
has a cutvertex $x^*$. Then $c(D^*-x) =2$:   
as if $c(D^*-x)  \ge 3$, then $c(G-(S\cup \{x,x^*\})) \ge 4$, implying $t\le 1$. 
Let 
$D^*_1$ and $D_2^*$ be the two components of 
$D^*-x^*$, 
and 
assume further that $|V(D^*_1)| \le |V(D^*_2)|$.  Then as $|V(D^*-x^*)|=4$, we have two possibilities: 
either $|V(D^*_1)|=1$ and $|V(D^*_2)|=3$ or $|V(D^*_1)|=2$ and $|V(D^*_2)|=2$.  Since $\delta(G) \ge 3$, if $|V(D^*_1)|=1$, then 
the vertex from $D^*_1$ must be adjacent in $G$ to at least one vertex from $S$. 
When 
$|V(D^*_2)|=3$ and $D^*_2 \ne K_3$, then $D^*_2$ has a cutvertex $x^{**}$. 
Then $S\cup \{x, x^*, x^{**}\}$ is a cutset of $G$ such that $c(G-(S\cup \{x, x^*, x^{**}\}))=5$, showing that $t\le 1$. 
Thus $D^*_2=K_3$.  The vertex $x^*$ is a cutvertex of $D^*$ and so is adjacent in 
$D^*$ to a vertex of $D_1^*$ and a vertex of $D_2^*$. 
However, this shows that $G$ contains $H_8$ as a spanning subgraph. 
When  $|V(D^*_2)|=2$,  
as  $G$  does not contain  $H_8$ or $H_9$ as a spanning subgraph, $x^*$
is adjacent in $G$ to exactly one vertex, say $x_1^*$,  of $D_1^*$
and to    exactly one vertex, say $x_2^*$,  of $D_2^*$.  
Then $S\cup \{x, x_1^*,x_2^*\}$ is a cutset of $G$ whose removal produces 5 components, 
showing that $\tau(G) \le 1$.

Assume then that $D^*$ is 2-connected. As $\delta(D^*) \le 2$, $D^*$ 
has a minimum cutset $W$ of size 2. If  $c(D^*-W)=3$, then we have $c(G-(S\cup W\cup \{x\}))=5$, showing that $t\le 1$. Thus $c(D^*-W)=2$. Then by analyzing the connection  in $D^*$ between $W$
and the two components of $D^*-W$, we see that $D^*$ contains $C_5$ as a spanning subgraph, 
 showing that $G$ contains $H_{10}$ as a spanning subgraph.

  \medskip 
 {\normalsize \scshape Case 1.1.2:} $|V(D)|=5$. 
 
Since $G$
does not contain $H_5$ as a spanning subgraph, we have $D \ne K_5$.  As $D \ne K_5$,  $D$
 has a cutset $W_D$ of size at most 3  such that each component of $D-W_D$ is a single vertex. 
 Then 
 \begin{align*}
 	t &\leq \frac{|S|+|W_D|}{|T|} \le \frac{4-2+3}{5}=1,   
 \end{align*}
a contradiction. 

{\normalsize \scshape Case 1.2:} $|\CC_3 \cup \CC_5|  \ge 2$.

By Equation (\ref{eqn1}), we have 
\begin{align*}
    4\geq |S| + \sum_{k\geq 1}k|\mathcal{C}_{2k+1}|.
\end{align*}
So one of the following holds:
\begin{enumerate}
    \item $S=\emptyset$ and either $|\mathcal{C}_5|\leq 2$, $|\mathcal{C}_5|\leq 1$ and $|\mathcal{C}_3|\leq 2$, or $|\mathcal{C}_3|\leq 4$. In this case, $\CC_1=\emptyset$ by Lemma~\ref{lem:NoC1s}(1). 
    Thus by Lemma~\ref{lem:biasedbarrier}(3), we have $e_G(T, V(G)\setminus T) \le 12<3|T|=15$. 
    \item $|S|=1$ and either $|\mathcal{C}_5|=1$ and $|\mathcal{C}_3|= 1$ or $|\mathcal{C}_3|\leq 3$. 
    In this case,  again $\CC_1=\emptyset$ by Lemma~\ref{lem:NoC1s}(1). 
    This implies there are a maximum of $14$ edges incident to vertices in $T$, a contradiction.
    \item $|S|=2$ and  $|\mathcal{C}_3|= 2$.  
    
    Let $\CC_3=\{D_1, D_2\}$. Note that $|V(D_i)| \ge 3$
    by Lemma~\ref{lem:biasedbarrier}(4) for each $i\in[1,2]$. 
    Since $|T|=5$, there exists $y_0\in T$
    such that $e_G(y_0, D_i)=1$ for each $i\in[1,2]$.  If $R=P_3\cup 4P_1$,  then $T$ together with the two neighbors of $y_0$ from $V(D_1)\cup V(D_2)$ induce $R$.
    If $R=6P_1$, then $T\setminus\{y_0\}$ together with the two neighbors of $y_0$ from $V(D_1)\cup V(D_2)$ gives an induced $6P_1$. If $R=7P_1$,  
    let $W_{D_i} \subseteq V(D_i)\setminus N_G(y_0)$ be the two vertices of $D_i$ 
    that are adjacent in $G$ to vertices from $T$.  Then $c(G-(S\cup W_{D_1} \cup W_{D_2} \cup \{y_0\}))=|T|-1+2=6$.
    Thus $t\le \frac{2+2+2+1}{6}=\frac{7}{6}$, contradicting $t>\frac{7}{6}$.  
    Lastly, assume $R=P_2\cup 5P_1$. If one of $D_i$ has at least 4 vertices, say $|V(D_2)| \ge 4$, then 
    let $x\in V(D_2)$ such that $e_G(x,T)=0$, $x^*\in V(D_1)$ and $y^*\in T$ such that $e_G(x^*,y^*)=1$.
    Then $x^*y^*$ and $(T\setminus\{y^*\}) \cup \{x\}$ induce $P_2\cup 5P_1$. 
    Thus $|V(D_1)|=|V(D_2)|=3$.  If one of $D_i$, say $D_2 \ne K_3$, then $D_2$
    has a cutvertex $x$. Let $W$ be the set of any two vertices of $D_1$. 
    Then $S\cup W\cup \{x\}$ is a cutset of $G$ such that $c(G-(S\cup W\cup \{x\}))=5$, showing that $t\le \frac{5}{5}=1$.
    Thus $D_1=D_2=K_3$. However, this shows that $G$ contains $H_{11}$ as a spanning subgraph. 
\end{enumerate}

\medskip 
{\normalsize \scshape Case 2:} $|T|=6$. 

In this case, by Lemma~\ref{lem:P4}(1), $G$ has an induced $P_4\cup 4P_1$, which contains each of $P_3\cup 4P_1, P_2\cup 5P_1$
and $6P_1$ as an induced subgraph. 
So we assume $R=7P_1$ in this case and thus $t>\frac{7}{6}$.  

Recall  for $y\in T$, $h(y)=|\{ D: \text{ $D\in \bigcup_{k\ge 1} \CC_{2k+1}$} \quad \text{and} \quad e_G(y,D) \ge 1\}|$.
If there exists $y_0\in T$ such that  $h(y_0) \ge 2$,  we let $x_1, x_2$ be the two neighbors of $y_0$ from the two corresponding components in $\bigcup_{k\geq 1}\mathcal{C}_{2k+1}$, respectively. Then $T\setminus\{y_0\}$ together with  $\{x_1,x_2\}$ induces $7P_1$. Thus $h(y) \le 1$
for each $y\in T$. This, together with $|T|=6$, implies that 
we have either $|\CC_3|\in \{1,2\}$ and $\CC_{2k+1}=\emptyset$ for any $k\ge 2$
or  $|\CC_5|=1$ and $\CC_{2k+1}=\emptyset$ for any $1 \le k \ne 2$. 

If $|\CC_3|=1$ and $\CC_{2k+1}=\emptyset$ for any $k\ge 2$, then $|S| \le 4$ by Equation \eqref{eqn1}.
Let $W$ be a set of two vertices  from the component in $\CC_3$  that are adjacent in $G$ to vertices from $T$. 
Then $c(G-(S\cup W)) \ge 6$, indicating that $t\le \frac{4+2}{6}<\frac{7}{6}$. 
For the other two cases, we have $|S| \le 3$.  If $|\CC_3|=2$ and $\CC_{2k+1}=\emptyset$ for any $k\ge 2$,
let $W$ be a set of four vertices, with two from one  component in $\CC_3$ and the other two from the other component
in $\CC_3$,   which are adjacent in $G$ to vertices from $T$.  If $|\CC_5|=1$ and $\CC_{2k+1}=\emptyset$ for any $1 \le k \le  2$,  let $W$ be a set of four vertices  from the component in $\CC_5$  that are adjacent in $G$ to vertices from $T$.  Then we have $c(G-(S\cup W)) \ge 6$, indicating that $t\le \frac{3+4}{6}=\frac{7}{6}$. 
\end{proof}

By Claim~\ref{clm5}, we now assume that $R \in \{P_7\cup 2P_1, P_5 \cup P_2, P_4 \cup P_3, P_3 \cup 2P_2, 3P_2 \cup P_1 \}$ and $t=3/2$. 
\begin{CLAIM}\label{clm2}
	There exists $y\in T$ with $h(y)>2$.
\end{CLAIM}
\pf Assume to the contrary that for every $y\in T$,   we have $h(y)\le 1$.
Define the following partition of $T$:
\begin{eqnarray*}
T_0 &=& \{y \in T : e_G(y, D) = 0 \text{ for all } D \in  \bigcup_{k\geq 1}\mathcal{C}_{2k+1}\},\\
T_1 &=& \{y \in T : e_G(y, D) = 1 \text{ for some } D \in  \bigcup_{k\geq 1}\mathcal{C}_{2k+1}\}. 
\end{eqnarray*}
Note that $|T_1|=\sum_{k\ge 1}(2k+1)|\mathcal{C}_{2k+1}|$ by Lemma~\ref{lem:biasedbarrier}(3)
and (4).  
For each $D\in \mathcal{C}_{2k+1}$ for some $k\ge 1$, we let $W_D$ be a set of $2k$
vertices that each has in $G$ a neighbor from $T$. As each $D-W_D$ is connected to exactly one 
vertex from $T$ and each component from $\CC_1$ is connected to exactly one 
vertex from $T$, it follows that $$W=S\cup  \bigcup_{D\in \bigcup_{k\geq 1}\mathcal{C}_{2k+1}} W_D$$ satisfies $c(G-W) \ge |T| \ge 5$, where $|T| \ge 5$ is by Claim~\ref{clm1}. 

By the toughness of $G$, we have 
\begin{eqnarray}
|S|+\sum_{k\ge 1}2k|\mathcal{C}_{2k+1}|&=&|W| \ge t|T|=t(|T_0|+|T_1|) \nonumber \\
 &=& t\left(|T_0|+\sum_{k\ge 1}(2k+1)|\mathcal{C}_{2k+1}| \right) \label{eqn3}.
 \end{eqnarray}
 
Since $t=3/2$, the inequality above implies that $|S|\ge 3|T_0|/2+\sum_{k\ge 1}(k+3/2)|\mathcal{C}_{2k+1}|$. Thus 
 \begin{eqnarray*}
 |S|+\sum_{k\ge 1}k|\mathcal{C}_{2k+1}| &\ge & 3|T_0|/2+\sum_{k\ge 1}(2k+3/2)|\mathcal{C}_{2k+1}|>|T_0|+\sum_{k\ge 1}(2k+1)|\mathcal{C}_{2k+1}|=|T|,
 \end{eqnarray*}
contradicting Equation~\eqref{eqn1}. 
\qed 

By Claim~\ref{clm2}, there exists $y\in T$ such that $h(y) \ge 2$. 
Then as $|T|\ge 5$, by Lemma~\ref{lem:P4}(2), 
$G$ contains an induced $P_7\cup 2P_1$. Thus 
we assume that $R \ne P_7\cup 2P_1$. 
We assume first that  $|\bigcup_{k\ge 1}\CC_{2k+1}| 
\geq 3$ and let $D_1, D_2, D_3$ be three distinct odd components from $\bigcup_{k\ge 1}\CC_{2k+1}$. 
Let $y_0\in T$ such that $h(y_0 )\ge 2$. We assume, without loss of generality, that $e_G(y_0,D_1)=e_G(y_0,D_2)=1$. 
By Lemma~\ref{lem:P4}(2), $G$ contains  an induced $P_b \cup aP_1$, where $b\ge 7$ and $a=|T|-3$, and the graph  $P_b \cup aP_1$ can be chosen such that the vertices in $aP_1$
are from $T$ and the path $P_b$ has the form $y_1x_1^*P_1x_1y_0x_2P_2x_2^*y_2$, where 
$y_0,y_1,y_2\in T$ and $x_1^*P_1x_1$ and $x_2^*P_2x_2$ are respectively contained in $D_1$ and $D_2$ such that $e_G(x,T)=0$
for every internal vertex $x$ from $P_1$ and $P_2$.  If one of $y_1$ and $y_2$, say $y_1$ has a neighbor $z_1$
from $V(D_3)$, 
then $z_1y_1x_1^*P_1x_1y_0x_2P_2x_2^*y_2$ and $T\setminus \{y_0, y_1, y_2\}$ induce 
$P_8\cup 2P_1$,  which contains each of $P_5 \cup P_2$,  $P_4 \cup P_3$,   and $3P_2 \cup P_1$
 as an induced subgraph. 
 Let $z_2 \in V(D_3)$ be a neighbor of $z_1$. Then $z_2z_1y_1x_1^*P_1x_1y_0x_2P_2x_2^*y_2$ 
 contains an induced $P_3\cup 2P_2$ whether $e_G(z_2, \{y_0, y_2\}) =0$ or $1$. 
 Thus we assume $e_G(y_i, D_3)=0$ for each $i\in [1,2]$ and so we can find $y_3\in T\setminus \{y_0, y_1,y_2\}$
 and $z\in V(D_3)$ such that $y_3z\in E(G)$. Then $y_1x_1^*P_1x_1y_0x_2P_2x_2^*y_2$ 
 and $zy_3$ contains an induced $P_7\cup P_2$, which contains each of $P_5 \cup P_2$, $P_3 \cup 2P_2$ and  $3P_2 \cup P_1$
 as an induced subgraph. 
 We are only left to consider $R=P_4\cup P_3$.  As $e_G(y_i, D_3)=0$ for each $i\in [1,2]$, 
 we can find distinct $y_3, y_4\in T\setminus\{y_0, y_1, y_2\}$ 
 and distinct $z_1, z_2\in V(D_3)$ such that $y_3z_1, y_4z_2\in E(G)$. 
 We let $P$ be a shortest path in $D_3$ connecting $z_1$ and $z_2$.  
 If $e_G(y_0, V(P))=0$, then $y_3z_1Pz_2y_4$ and $y_1x_1^*P_1x_1y_0x_2P_2x_2^*y_2$ 
 contains an induced $P_4 \cup P_3$. Thus $e_G(y_0, V(P))=1$. This in particular, implies that $|V(P)| \ge 3$. 
 Then $y_3z_1Pz_2y_4$ and $y_1x_1^*P_1x_1$ together contain an induced $P_4 \cup P_3$.

Thus $|\cup_{k\geq 1}\mathcal{C}_{2k+1}|= 2$. 
Let $D_1, D_2 \in \bigcup_{k\geq 1}\mathcal{C}_{2k+1}$ be the two components. Define the following partition of $T$:
\begin{eqnarray*}
T_0 &=& \{y \in T : e_G(y, D_1) =e_G(y,D_2)=0\},\\
T_{11} &=& \{y \in T : e_G(y, D_1) = 1 \text{ and } e_G(y, D_2) = 0\},\\
T_{12} &=& \{y \in T : e_G(y, D_1) = 0 \text{ and } e_G(y, D_2) = 1\},\\
T_2 &=& \{y \in T : e_G(y, D_1) =e_G(y, D_2) = 1\}. 
\end{eqnarray*}

We have either $T_2 = \emptyset$ or $T_2 \neq \emptyset$.
First suppose $T_2 = \emptyset$. Define the following vertex sets:
$$W_1 = N_G(T_{11}) \cap V(D_1) \quad \text{and} \quad 
W_2 = N_G(T_{12}) \cap V(D_2).$$
Then $|W_1| = |T_{11}| = 2k_1+1$ and $|W_2| = |T_{12}| = 2k_2+1$, where we assume $e_G(T, D_1)=2k_1+1$ and  $e_G(T, D_2)=2k_2+1$ for some integers $k_1$ and $k_2$. Then $W = S \cup W_1 \cup W_2$ is a cutset of $G$ with $c(G-W) \geq |T|$. By toughness, $|W| \geq \frac{3}{2}|T| = |T| + \frac{1}{2}|T|$. Since $|T| = |T_0| + |T_{11}| + |T_{12}|$, this gives us 
\begin{align*}
    |W| \ & \geq |T| + \frac{1}{2}|T_0| + \frac{1}{2}(|T_{11}| + |T_{12}|) \\
    & = |T| + \frac{1}{2}|T_0| + \frac{1}{2}(2k_1 + 1 + 2k_2+1) \\
     & = |T| + \frac{1}{2}|T_0| + k_1 + k_2 +1.
\end{align*}
Thus  $|W|=|S|+|W_1|+|W_2|=|S| + 2k_1 + 2k_2 + 2 \geq |T| + \frac{1}{2}|T_0| + k_1 + k_2 +1$, which implies $|S| + k_1 + k_2 +1 \geq |T| + \frac{1}{2}|T_0|$. Hence, by Equation (\ref{eqn1}), we have $|T| \geq |T| + \frac{1}{2}|T_0|$, giving  a contradiction.

So we may assume $T_2 \neq \emptyset$. 
Now define the following vertex sets:
$$W_1 = N_G(T_{11}) \cap V(D_1), \quad  W_2 = N_G(T_{12}) \cap V(D_2), \quad \text{and} \quad 
W_3 = N(T_2) \cap (V(D_1) \cup V(D_2)).$$

We have that $|W_1| = |T_{11}|$, $|W_2| = |T_{12}|$, and $|W_3| = 2|T_2|$. Now let $W = S \cup W_1 \cup W_2 \cup W_3$. Then $W$ is a cutset  of $G$ with $c(G-W) \geq |T_0|+|T_{11}|+|T_{12}|+1$ since $T_2 \neq \emptyset$. By toughness, $|W| \geq \frac{3}{2}(|T_0|+|T_{11}|+|T_{12}|+1)$. Since $|W| = |S| + |W_1| + |W_2|+ |W_3| = |S| + |T_{11}| +|T_{12}| +2|T_2|$, we have $|S| + |T_{11}| +|T_{12}| +2|T_2| \geq \frac{3}{2}|T_0|+ \frac{3}{2}|T_{11}|+ \frac{3}{2}|T_{12}|+\frac{3}{2}$. This implies
$$|S| \geq \frac{3}{2}|T_0| + \frac{1}{2}|T_{11}| + \frac{1}{2}|T_{12}| +1.$$
Thus, 
\begin{equation} \label{eq:sizeofSwithks}
    |S| +k_1 +k_2 \geq \frac{3}{2}|T_0| + \frac{1}{2}|T_{11}| + \frac{1}{2}|T_{12}| +1 +k_1 +k_2.  
\end{equation}

We have that either $T_{11}\cup T_{12}\cup T_0 = \emptyset$ or $T_{11}\cup T_{12} \cup T_0 \neq \emptyset$.
First suppose $T_{11} \cup T_{12} \cup T_0 = \emptyset$. Then $|T| = |T_2| = \frac{1}{2}(2k_1+1+2k_2+1) = k_1+k_2+1$.  Thus $|S|+k_1+k_2 \ge |T|$, 
showing a contradiction to Equation~(\ref{eqn1}). 


So we may assume $T_{11} \cup T_{12} \cup T_0 \neq \emptyset$. Then 
\begin{align*}
    |T| & =  |T_0| + (2k_1 +1 +2k_2 + 1 - |T_2|)\\
    &= |T_0| + (2k_1 + 2k_2 +2) - \frac{1}{2}(2k_1 +1 + 2k_2 + 1 - |T_{11}| - |T_{12}|) \\
    &= |T_0| + \frac{1}{2}(2k_1 + 2k_2 +2) + \frac{1}{2}|T_{11}| + \frac{1}{2}|T_{12}| \\
    &= |T_0| + k_1 + k_2 + 1 + \frac{1}{2}|T_{11}| + \frac{1}{2}|T_{12}|.  
\end{align*}
Using the size of $T$ and~\eqref{eq:sizeofSwithks}, we get $|S| +k_1 + k_2 \geq |T|$, showing a contradiction to Equation~(\ref{eqn1}). 

The proof of Theorem~\ref{thm9} is now finished. 
\qed

\section{Proof of Theorems~\ref{theorem4a} and~\ref{theorem4}}

Recall that for a graph $G$, $\alpha(G)$, the independence number of $G$, is the size of a largest independent set in $G$. 
\proof[Proof of Theorem~\ref{theorem4a}]

For each $i\in [0,11]$, $H_i$ does not contain a 2-factor by Theorem~\ref{tutte's theorem}. 
Thus to finish proving Theorem~\ref{thm9}, 
we are only left to show the three claims below. 

\begin{CLAIM}
	The graph $H_i$ is $(P_2\cup 3P_1)$-free, $H_1$ is  $(P_3\cup 2P_1)$-free,
	and $\tau(H_i)=1$ for each $i\in [0,4]$.  	
\end{CLAIM}

\pf  We first show that  $H_i$ is $(P_2\cup 3P_1)$-free 
 for each $i\in [0,4]$. 
We only show this for $H_0$, as  the proofs for $H_i$ for $i \in [1,4]$ are similar.
In $H_0$, there are two types of edges $xy$: $ x, y \in V(D_j) \text{ or } x\in V(D_j) \text{ and } y \in V(T)$, where $j\in [1,2]$. Without loss  of generality first consider the edge $v_1v_2\in E(D_1)$ and the subgraph $F_1=H_0 - (N_{H_0}[v_1] \cup N_{H_0}[v_2])$. We see  $\alpha(F_1)=2$.  Now, without loss of generality, consider the edge $v_1t_1$ and the subgraph $F_2=H_0 - (N_{H_0}[v_1] \cup N_{H_0}[t_1])$. We see $\alpha(F_2)=2$. In either case, $P_2 \cup 3P_1$ cannot exist as an induced subgraph in $H_0$. Thus $H_0$ is $(P_2 \cup 3P_1)$-free.

Then we show that $H_1$ is $(P_3\cup 2P_1)$-free. Two types of induced paths  $abc$ of length 3 exist: $ a\in S, b \in T, c \in V(D) \text{ or } a \in T, b,c \in V(D)$. Without loss of generality, consider the path $xt_1v_1$ and the subgraph $F_1 = H_1 - (N_{H_1}[x] \cup N_{H_1}[t_1] \cup N_{H_1}[v_1])$. We see that $F_1$ is a null graph. Now, without loss of generality, consider the path $t_1v_1v_2$ and the subgraph $F_2 = H_1 -  (N_{H_1}[t_1] \cup N_{H_1}[v_1] \cup N_{H_1}[v_2])$. We see $|V(F_2)|=1$. In either case, $P_3 \cup 2P_1$ cannot exist as an induced subgraph in $H_1$. Thus $H_1$ is $(P_3 \cup 2P_1)$-free.

Let $i\in [0,4]$. As $\delta(H_i)=2$, $\tau(H_i) \le 1$. 
It suffices to show $\tau(H_i) \ge 1$. Since $H_i$ is 2-connected,  we show 
that $c(H_i-W) \le |W|$ for any  $W\subseteq V(H_i)$ such that $|W| \ge 2$.  If $|W|=2$, by considering all the possible formations 
of $W$,  
we have $c(H_i-W) \le |W|$. Thus we assume $|W| \ge 3$.

Assume by contradiction that there exists $W\subseteq V(H_i)$ with $|W| \ge 3$ and $c(H_i-W) \ge  |W|+1\ge 4$. 
The size of a largest independent set of each $H_0$, $H_2$, $H_3$, and $H_4$ is 4, and of $H_1$ is 3. Since $c(H_i-W)$ is bounded above by the 
size of  a largest independent set of $H_i$, 
we already obtain a contradiction if $i=1$ or $|W| \ge 4$. So we assume $i\in \{0,2,3,4\}$ and $|W|=3$. 

As $c(H_i-W) \ge 4$, for the graph $H_0$, 
we must have $\{v_1,v_2, v_3\} \cap W \ne \emptyset$ 
and $\{v_4,v_5, v_6\} \cap W \ne \emptyset$. 
As $|W|=3$, we have either $W\cap T=\emptyset$
or $|W\cap T|=1$. In either case, by checking all the possible formations of $W$,  we get $c(H_0-W) \le 2$, 
contradicting the choice of $W$.

As $c(H_i-W) \ge 4$, for each $i\in [2,4]$, we must have $x\in W$. 
Thus $t_j\not\in W$ for $j\in [1,3]$, as otherwise, $c(H_i-(W\setminus \{t_j\})) \ge 4$, contradicting the argument previously that $c(H_i-W^*) \le 2$
for any $W^*\subseteq V(H_i)$ and $|W^*| \le 2$. 
As $|W|=3$, we then have $|W\cap \{v_1,v_2,v_3, v_4\}|=2$. 
However, $c(H_i-W)\le 3$ for $W=\{x, v_k,v_\ell\}$
for all distinct $k,\ell\in [1,4]$.  We again get a contradiction to the  choice of $W$.
\qed

\begin{CLAIM}
	The graph $H_5$ with $p=5$ is $(P_3\cup 4P_1)$-free, $(P_2\cup 5P_1)$-free, and $6P_1$-free with $\tau(H_5)=\frac{6}{5}$.  	
\end{CLAIM}

\pf  Let $p=5$ and $D$ be the odd component of $H_5-(S\cup T)$. Note that $D=K_p=K_5$.  

We first show that $H_5$ is $(P_3 \cup 4P_1)$-free. 
There are three types of induced paths $xyz$ of length 3 in $H_5:  x \in S, y \in T, z \in V(D) \text{ or } x \in T, y,z \in V(D) \text{ or } x,z \in T, y \in S$. Without loss of generality, consider the path $x_1t_1y_1$ and the subgraph $F_1 = H_5 - (N_{H_5}[x_1] \cup N_{H_5}[t_1] \cup N_{H_5}[y_1])$. We see that $F_1$ is a null graph. Now consider the path $t_1y_1y_2$ and the subgraph $F_2 = H_5 -(N_{H_5}[t_1] \cup N_{H_5}[y_1] \cup N_{H_5}[y_2])$. We see $\alpha(F_2) = 3$. Finally consider the path $t_1x_1t_2$ and the subgraph $F_3 = H_5 - (N_{H_5}[t_1] \cup N_{H_5}[x_1] \cup N_{H_5}[t_2])$. We see $\alpha(F_3) = 3$. In any case, an induced copy of $P_3 \cup 4P_1$ cannot exist in $H_5$. Thus $H_5$ is $(P_3 \cup 4P_1)$-free.

We then show that $H_5$ is $(P_2 \cup 5P_1)$-free. 
There are three types of edges $xy$ in $H_5:  x \in S, y \in T \text{ or } x \in T, y \in V(D) \text{ or } x,y \in V(D)$. Without loss of generality, consider the edge $x_1t_1$ and the subgraph $F_1 = H_5 - (N_{H_5}[x_1] \cup N_{H_5}[t_1])$. We see $|V(F_1)| = 4$. Now consider the edge $t_1y_1$ and the subgraph $F_2 = H_5 - (N_{H_5}[t_1] \cup N_{H_5}[y_1])$. We see $|V(F_2)| = 4$. Finally, consider the edge $y_1y_2$ and the subgraph $F_3 = H_5 - (N_G{H_5}[y_1] \cup N_{H_5}[y_2])$. We see $\alpha(F_3) = 3$. In any case, no induced copy of $P_2 \cup 5P_1$ can exist in $H_5$. Thus $H_5$ is $(P_2 \cup 5P_1)$-free.

We lastly show that $H_5$ is $6P_1$-free.
There are three types of vertices  $x$ in $H_5:  x \in S, x \in T$,  or $x \in V(D)$. Without loss of generality, consider the vertex $x_1$ and the subgraph $F_1 = H_5 - N_{H_5}[x_1]$. We see $\alpha(F_1) =1$. Now consider the vertex $t_1$ and the subgraph $F_2 = H_5 -  N_{H_5}[t_1]$. We see $\alpha(F_2) =4$. Finally, consider the vertex $y_1$ and the subgraph $F_3 = H_5 -  N_{H_5}[y_1]$. We see $\alpha(F_3) =4$. In any case, no induced copy of $6P_1$ can exist in $H_5$. Thus $H_5$ is $6P_1$-free.

We now show that $\tau(H_5)=\frac{6}{5}$.
Let $W$ be a toughset of $H_5$. 
Then $S\subseteq W$. Otherwise, by the structure of $H_5$, we have $c(H_5-W) \le 3$ and $|W| \ge 5$. 
As  $S\subseteq W$ and the only neighbor of each vertex of $T$
in $H_5-S$ is contained in a clique of $H_5$, we have 
$T\cap W=\emptyset$.  Since $c(H_5-W) \ge 2$, it follows that 
$W\cap V(D) \ne \emptyset$. 
Then $c(H_5-W)=|W\cap V(D)|$
if $|W\cap V(D)| \le 3$ or $|W\cap V(D)| =5$,  and 
$c(H_5-W)=|W\cap V(D)|+1$ if $|W\cap V(D)| =4$. 
The smallest ratio of $\frac{|W|}{c(H_5-W)}$ is $\frac{6}{5}$,
which happens when $|W\cap V(D)| =4$. 
\qed 

\begin{CLAIM}
	The graph $H_i$ is $(P_2\cup 5P_1)$-free with $\tau(H_i)=\frac{7}{6}$ for each $i\in [6,11]$.   	
\end{CLAIM}

\pf   We show first that each $H_i$ is $(P_2\cup 5P_1)$-free. We do this only 
for the graph $H_6$, as the proofs for the rest graphs are similar. 
For any edge $ab \in E(H_6)$, we see $\alpha(H_6 - (N_{H_6}[a] \cup  N_{H_6}[b])) \leq 4 $. Thus no induced copy of $(P_2\cup 5P_1)$ can exist in $H_6$. Thus $H_6$ is $(P_2 \cup 5P_1)$-free.

We next show that $\tau(H_i)=\frac{7}{6}$ for each $i\in [6,10]$. 
We have $c(H_i-(S\cup \{v_1, \ldots, v_5\}))=6$,  implying $\tau(H_i) \le \frac{7}{6}$. Suppose $\tau(H_i) < \frac{7}{6}$. Let $W$ be a toughset of $H_i$. As each $H_i$ is 3-connected, we have 
$|W| \ge 3$.  Thus $c(H_i-W) \ge 3$. 
We have that either $S \subseteq W$ or $S \not\subseteq W$. Suppose the latter. Then we have $S \cap V(H_i-W) \ne \emptyset$.  Then all vertices in $T\setminus W$ are contained in the same component as the one which contains  $S\setminus W$.  Since $c(H_i-W) \ge 3$, by the structure of $H_i$, it follows that we have either $T\subseteq W$ or $\{v_1, \ldots, v_5\} \subseteq W$. In either case, we have $c(H_i-W) \le 3$, implying $\frac{|W|}{c(H_i-W)}  \ge  \frac{5}{3} >\frac{7}{6}$, a contradiction. So $S \subseteq W$. By Lemma~\ref{lem:tough-set}, $t_j \not\in W$ for all $j\in [1,5]$. Thus each $t_j \in V(H_i-W)$. Now either $v_0 \in W$ or $v_0 \not\in W$. Suppose $v_0 \in W$, then we cannot have all $v_j \in W$ without violating Lemma~\ref{lem:tough-set}. In this case, the minimum ratio  $\frac{|W|}{c(H_i-W)}$ occurs when $|W \cap \{v_1, v_2, v_3, v_4, v_5\}| =3$. This implies $\frac{|W|}{c(H_i-W} \geq \frac{6}{5} > \frac{7}{6}$, a contradiction. Thus $v_0 \not\in W$ and we must have $v_0 \in V(H_i-W)$. This implies $\{v_1 \ldots v_5\} \subseteq W$ and $\frac{|W|}{c(H_i-W)} = \frac{7}{6}$, a contradiction.
Thus $\tau(H_i)=\frac{7}{6}$ for each $i\in [6,10]$. 


Lastly we show $\tau(H_{11})=\frac{7}{6}$. We have $c(H_{11}-(S\cup \{v_1, v_2, t_3, v_4, v_5\}))=6$,  implying $\tau(H_{11}) \le \frac{7}{6}$. Suppose $\tau(H_{11}) < \frac{7}{6}$. Let $W$ be a tough set of $H_{11}$. As $H_{11}$ is 3-connected, we have $|W| \ge 3$.
Thus $c(H_{11}-W) \ge 3$.  We have that either $S \subseteq W$ or $S \not\subseteq W$. Suppose the latter. Then we have $S \cap V(H_{11}-W) \ne \emptyset$.  Then all vertices in $T\setminus W$ are contained in the same component as the one which contains  $S\setminus W$.  Since $c(H_{11}-W) \ge 3$, by the structure of $H_{11}$, it follows that $|W| \ge 5$ and $c(H_{11}-W) \le 4$. This implies $\frac{|W|}{c(H_{11}-W)}  \ge  \frac{5}{4} >\frac{7}{6}$, a contradiction. So $S \subseteq W$. By Lemma~\ref{lem:tough-set}, $t_i \not\in W$ for $i \in \{1,2,4,5\}$. Thus $t_i \in V(H_{11}-W)$ for $i \in \{1,2,4,5\}$ and we must have $W \cap \{v_1,v_2,v_3,v_4,v_5,v_6,t_3\}\ne \emptyset$. If $t_3 \not\in W$, then $\frac{|W|}{c(H_{11-W})} \ge \frac{6}{5}>\frac{7}{6}$, a contradiction. 
Thus $t_3 \in W$. Then $v_3$ and $v_4$ are respectively in two distinct components of $H_{11}-W$
by Lemma~\ref{lem:tough-set}.  Thus $W\cap \{v_1, v_2, v_5, v_6\} \ne \emptyset$
as $c(H_{11}-W) \ge 3$.  Furthermore, we have $c(H_{11}-W)=|W\cap \{v_1, v_2, v_5, v_6\}|+2$. 
The smallest ratio of $\frac{|W|}{c(H_{11}-W)}$ is $\frac{7}{6}$,
which happens when $\{v_1, v_2, v_5, v_6\} \subseteq W$. Again we get a contradiction to 
the assumption that   $\tau(H_{11}) < \frac{7}{6}$. 
Thus $\tau(H_{11}) = \frac{7}{6}$.
\qed 

The proof of Theorem~\ref{thm9} is complete. 
\qed 

\proof[Proof of Theorem~\ref{theorem4}] Let $p \ge 6$ and $D$ be the odd component of $H_5-(S\cup T)$. Note that $D=K_p$.  Since $c(H_5-(S\cup \{y_1, \ldots, y_5\}))=6$,  we have 
$\tau(H_5) \le \frac{7}{6}$. 
We show $\tau(H_5)  \ge  \frac{7}{6}$.
Let   $W$ be a toughset of $H_5$.  Then either $S \subseteq W$ or $S \not\subseteq W$. Suppose the latter. Then we have $S \cap V(H_5-W) \ne \emptyset$.  Then all vertices in $T\setminus W$ are contained in
 the same component as the one containing  $S\setminus W$.  Since $c(H_5-W) \ge 2$, by the structure of $H_5$, it follows that we have either $T\subseteq W$ or $\{y_1, \ldots, y_5\} \subseteq W$. In either case, we have $c(H_5-W) \le 3$, 
 implying $\frac{|W|}{c(H_5-W)}  \ge  \frac{5}{3} >\frac{7}{6}$.  
 Now suppose $S \subseteq W$. By Lemma~\ref{lem:tough-set}, $t_i \not\in W$ for all $i$. Thus each $t_i \in V(H_5-W)$.
 Furthermore, $c(H_5-W)=|W\cap V(D)|+1$. Since $W$ is a cutset of $G$, 
 we have $|W\cap V(D)| \ge 2$. The smallest ratio of $\frac{|W|}{c(H_5-W)}$ is $\frac{7}{6}$,
 which happens when $|W\cap V(D)| =5$.

For the graph $H_{12}$, we have $c(H_{12} - (S\cup \{y_1, y_2, y_3\}))=4$,  implying 
$\tau(H_{12}) \le \frac{4}{4} = 1$. 
We show $\tau(H_{12})  \ge  1$.
Let $W$ be a toughset of $H_{12}$.  Then either $S \subseteq W$ or $S \not\subseteq W$. Suppose the latter. Then we have $S \cap V(H_{12}-W) \ne \emptyset$.  Then all vertices in $T\setminus W$ are contained in
 the same component as the one containing  $S\setminus W$.  Since $c(H_{12}-W) \ge 2$, by the structure of $H_{12}$, it follows that we have either $T\subseteq W$ or $\{y_1, y_2, y_3\} \subseteq W$. In either case, we have $c(H_{12}-W) \le 2$, 
 implying $\frac{|W|}{c(H_{12}-W)}  \ge  \frac{3}{2} > 1$.
 Now suppose $S \subseteq W$. By Lemma~\ref{lem:tough-set}, $t_i \not\in W$ for all $i$. Thus each $t_i \in V(H_{12}-W)$. This implies $|\{y_1, y_2, y_3\} \cap W| = 2 \text{ or } 3$. In either case we see $\frac{|W|}{c(H_{12}-W)} = 1$.
\qed

\bibliographystyle{plain}
\bibliography{bibliography}

\end{document}